\documentclass[12pt]{amsart}
\usepackage{amssymb,amsmath}
\usepackage[ps2pdf,colorlinks=true,urlcolor=blue,
citecolor=red,linkcolor=blue,linktocpage,pdfpagelabels,bookmarksnumbered,bookmarksopen]{hyperref}

\usepackage[left=3.2cm,right=3.2cm,top=3.2cm,bottom=3.2cm]{geometry}

\newcommand{\N}{{\mathbb N}}

\newcommand{\R}{{\mathbb R}}
\newcommand{\C}{{\mathbb C}}

\newcommand{\eps}{\varepsilon}

\newcommand{\im}{{\mathcal Im}}

\makeatother

\title[Stability and instability for quasi-linear
Schr\"odinger equations]{Stability and instability results for standing \\ waves of
quasi-linear Schr\"odinger equations}

\author{Mathieu Colin}
\address{
Math\'ematiques Appliqu\'ees de Bordeaux
\newline\indent
Universit\'{e} Bordeaux I
\newline\indent
351, cours de la Lib\'{e}ration 33405 Talence Cedex, France}
\email{Mathieu.Colin@math.u-bordeaux1.fr}

\author{Louis Jeanjean}
\address{Laboratoire de Math\'ematiques (UMR 6623)
\newline\indent
Universit\'{e} de Franche-Comt\'{e}
\newline\indent
16, Route de Gray 25030 Besan\c{c}on Cedex, France}
\email{louis.jeanjean@univ-fcomte.fr}

\author{Marco Squassina}
\address{Dipartimento di Informatica
\newline\indent
Universit\`a degli Studi di Verona
\newline\indent
C\'a Vignal 2, Strada Le Grazie 15, 37134 Verona, Italy}
\email{marco.squassina@univr.it}

\thanks{The third author was partially supported by the
Italian PRIN Research Project 2007 {\em Metodi variazionali e topologici
nello studio di fenomeni non lineari}}

\newtheorem{theorem}{Theorem}[section]
\newtheorem{proposition}[theorem]{Proposition}

\newtheorem{lemma}[theorem]{Lemma}
\theoremstyle{definition}

\newtheorem{remark}[theorem]{Remark}

\numberwithin{equation}{section}

\newcommand{\iu}{{\rm i}}

\newcommand{\V}{{\mathbb V}}

\newcommand{\lapl}{\Delta}

\begin{document}

\subjclass[2000]{35J40; 58E05}

\keywords{Quasi-linear Schr\"odinger equations, orbital stability, orbital instability by blow-up,
ground state solutions, variational methods}

\begin{abstract}
We study a class of quasi-linear Schr\"odinger equations arising in the theory
of superfluid film in plasma physics. Using gauge transforms and a derivation
process we solve, under some regularity assumptions, the Cauchy problem. Then,
by means of variational methods, we study the existence, the orbital stability
and instability of standing waves which minimize some associated energy.
\end{abstract}

\maketitle

\medskip
\begin{center}
\begin{minipage}{11cm}
\footnotesize
\tableofcontents
\end{minipage}
\end{center}

\medskip

\section{Introduction and main results}
Several physical situations are described by generic quasi-linear
equations of the form
\begin{equation}
    \label{eq.schr0-intro}
  \begin{cases}
  \iu\phi_t+\lapl\phi+\phi \ell^{'}(|\phi|^2)\lapl \ell( |\phi|^2)+f(|\phi|^2)\phi=0 & \text{in $(0,\infty)\times \R^N$},\\
  \phi(0,x)=a_0(x) & \text{in $\R^N$},
  \end{cases}
\end{equation}
where $\ell$ and $f$ are given functions. Here $\iu$ is the
imaginary unit, $N\geq 1$, $\phi:\R^N\to\C$ is a complex valued
function. For example, the particular case $\ell(s)=\sqrt{1+s}$ models the
self-channeling of a high-power ultra short laser in matter
(see~\cite{boga,deB,ritchie}) whereas if $\ell(s)=\sqrt{s}$,
equation~\eqref{eq.schr0-intro} appears in dissipative quantum
mechanics (\cite{hasse}). It is also used in plasma physics and
fluid mechanics (\cite{goldman,litvak}), in the theory of
Heisenberg ferromagnets and magnons (\cite{bass}) and in condensed
matter theory (\cite{makhankov}). The dynamical features are closely related
to the two functions $\ell$ and $f$. Only few intents have been
done to develop general theories for the Cauchy problem (see
nevertheless~\cite{co,kpv,poppenberg}). In this article we focus
on the particular case $\ell(s)=s$, that is
\begin{equation}
    \label{eq.schr1-intro}
  \begin{cases}
  \iu\phi_t+\lapl\phi+\phi\lapl |\phi|^2+f(|\phi|^{2})\phi=0 & \text{in $(0,\infty)\times \R^N$},\\
  \phi(0,x)=a_0(x) & \text{in $\R^N$}.
  \end{cases}
\end{equation}
 Our first result concerns the
Cauchy problem. Due to the quasi-linear term, it seems difficult
to exhibit a well-posedness result in the natural energy space
$$
X_\C =\Big\{u \in H^1(\R^N,\C): \int_{\R^N} |u|^2 |\nabla |u||^2
dx<\infty \Big\}.
$$
The local and global well-posedness of the Cauchy
problem~\eqref{eq.schr0-intro} have been studied by Poppenberg
in~\cite{poppenberg} in any dimension $N\geq 1$ and for smooth
initial data, precisely belonging to the space $H^\infty$. In
\cite{co}, equation \eqref{eq.schr0-intro} is solved locally in
the function space $L^\infty(0,T;H^{s+2}(\R^N))\cap C([0,T);H^s(\R^N)),$
where $s=2E(\frac{N}{2}) +2$ (here $E(a)$ denotes the integer part
of $a$) for any initial data and smooth nonlinearities $\ell$ and
$f$ such that there exists a positive constant $C_\ell$ with
\begin{align}\label{cond_co}
1-4\sigma\ell^{'2}(\sigma)> C_\ell \ell^{'2}(\sigma),\quad\text{for all $\sigma\in\R_+$}.
\end{align}
Note that the function $\ell(\sigma)=\sigma$ does not
satisfied~\eqref{cond_co} and, then, it is not possible to apply
\cite[Theorem 1.1]{co} to problem~\eqref{eq.schr1-intro}. Before stating
our result, we introduce the energy functional ${\mathcal E}$
associated with~\eqref{eq.schr1-intro}, by setting
$$
{\mathcal E}(\phi)=\frac{1}{2}\int_{\R^N}|\nabla
\phi|^2dx+\frac{1}{4}\int_{\R^N}|\nabla |\phi|^2|^2dx
-\int_{\R^N}F(|\phi|^2)dx,
$$
for all $\phi\in X_{\C}$, where $F(\sigma)=\int_0^\sigma f(u)du$.
Note that ${\mathcal E}(\phi)$ can also be written
$$
{\mathcal E}(\phi)=\frac{1}{2}\int_{\R^N}|\nabla
\phi|^2dx+\int_{\R^N}|\phi|^2|\nabla |\phi||^2dx
-\int_{\R^N}F(|\phi|^2)dx.
$$
We prove the following.
\begin{theorem}\label{cauchy}
Let $N\geq 1$, $s= 2E(\frac{N}{2})+2$ and assume that $a_0\in H^{s+2}(\R^N)$
and $f \in C^{s+2}(\R^N)$. Then there exists a positive $T$ and a unique
solution to the Cauchy problem $\eqref{eq.schr1-intro}$ satisfying
$$\phi(0,x)=a_0(x),$$
$$\phi \in L^{\infty}(0,T;H^{s+2}(\R^N))\cap C([0,T];H^s(\R^N)),$$
and the conservation laws
\begin{align}
    \|\phi(t)\|_{2} & =\|a_0\|_{2},   \label{massc}\\
    {\mathcal E}(\phi(t)) & ={\mathcal E}(a_0),\label{energyc}
\end{align}
for all $t\in [0,T]$.
\end{theorem}
The proof of Theorem \ref{cauchy} follows the approach developed
in~\cite{co}. It is based on energy methods and to overcome the
loss of derivatives induced by the quasi-linear term, gauge
transforms are used. We rewrite equation~\eqref{eq.schr0-intro} as
a system in $(\phi,\overline{\phi})$ where $\overline{z}$ denotes
the complex conjugate of $z$. Then, we differentiate the resulting
equation with respect to space and time in order to linearize the
quasi-linear part and we introduce a set of new unknowns
(see~\eqref{new_unkn}). A fixed-point procedure is then applied on the
linearized version. Since~\eqref{cond_co} does not hold we need,
with respect to~\cite{co}, to modify the linearized version and
to perform different energy estimates on the Schr\"odinger part of
the equation.
\medskip

Next, equipped with Theorem~\ref{cauchy} and motivated by~\cite{BoCh} we
investigate some questions of existence, stability and instability of standing
waves solutions of~\eqref{eq.schr1-intro}, when $f$ is the power nonlinearity
$f(s)=s^{\frac{p-1}{2}}$, with $p>1$. In this case~\eqref{eq.schr1-intro}
becomes
\begin{equation}
    \label{eq.schr1}
  \begin{cases}
  \iu\phi_t+\lapl\phi+\phi\lapl |\phi|^2+|\phi|^{p-1}\phi=0 & \text{in $(0,\infty)\times \R^N$},\\
  \phi(0,x)=a_0(x) & \text{in $\R^N$}.
  \end{cases}
\end{equation}
If $p>1$ is an odd integer or $p> 4 E(\frac{N}{2}) + 9$ then $f(s) =
s^{\frac{p-1}{2}}$ belongs to $C^{2+s}(\R^N)$ and thus Theorem~\ref{cauchy}
applied. In the remaining cases, when we state our stability or instability
results, we shall always assume that there exists a solution to the Cauchy
problem~\eqref{eq.schr1} for our nonlinearity and our initial data $a_0 \in
X_{\C}$.
\medskip

By standing waves, we mean solutions of the form $\phi_\omega
(t,x)=u_\omega (x)e^{-i\omega t}$. Here $\omega$ is a fixed
parameter and $\phi_\omega (t,x)$ satisfies problem~\eqref{eq.schr1-intro}
if and only if $u_\omega$ is a solution of the equation
\begin{equation}\label{gs1}
    -\lapl u-u\lapl (|u|^2)+\omega u =|u|^{p-1}u,\qquad  \text{in $\R^N$}.
\end{equation}
Throughout the paper we assume that $1<p< \frac{3N+2}{N-2}$ if $N
\geq 3$ and $p>1$ if $N=1,2$. A function $u \in X_\C$ is called a
(complex) weak solution of equation~\eqref{gs1} if
\begin{equation} \label{weak-solution}
\Re\int_{\R^N}\Big(\nabla u\cdot \nabla \overline{\phi} + \nabla
(|u|^2)\cdot \nabla(u \overline{\phi}) + \omega u \overline{\phi}
- |u|^{p-1}u \overline{\phi}\Big) \, dx =0
\end{equation}
for all $\phi \in C_0^{\infty}(\R^N,\C)$ (here $\Re(z)$ is the real part of $z\in\C$).
We say that a weak solution of~\eqref{gs1}
is a ground state if it satisfies
\begin{eqnarray}\label{defgs}
\mathcal{E}_\omega(u)=m_\omega,
\end{eqnarray}
where
\begin{equation*}
m_{\omega}=\inf\{ {\mathcal E}_{\omega}(u): \text{$u$ is a
nontrivial weak solution of~\eqref{gs1}}\}.
\end{equation*}
Here, $\mathcal{E}_\omega$ is the action associated
with~\eqref{gs1} and reads
$$
{\mathcal E}_{\omega}(u) = \frac{1}{2}\int_{\R^N} |\nabla u|^2 dx
+ \frac{1}{4}\int_{\R^N}  |\nabla |u|^2|^2 dx+\frac{\omega}{2}\int_{\R^N}
|u|^2 dx - \frac{1}{p+1}\int_{\R^N} |u|^{p+1}dx.
$$
We denote by ${\mathcal G_{\omega}}$ the set of weak solutions
to~\eqref{gs1} satisfying~\eqref{defgs}. It is easy to check that
$u$ is a weak solution of equation~\eqref{gs1} if, and only if,
$$
{\mathcal E}_{\omega}'(u) \phi := \lim_{t \to 0^+}\frac{{\mathcal E}_{\omega}(u+t\phi) - {\mathcal
E}_{\omega}(u)}{t}=0,
$$
for every direction $\phi \in C_0^{\infty}(\R^N,\C)$.

\medskip

First we establish the existence of ground states to~\eqref{gs1}
and we derive some qualitative properties of the elements of
${\mathcal G_{\omega}}$. Our existence result complements the ones
of~\cite{cojean,LiWaWa1,LiWaWa2,LiWa,poppenberg2}.

\begin{theorem}
    \label{ex}
For all $\omega>0$, ${\mathcal G_{\omega}}$ is non void and any
$u \in {\mathcal G_{\omega}}$ is of the form
\begin{equation*}
u(x)=e^{i\theta}|u(x)|,\quad x\in\R^N,
\end{equation*}
for some $\theta\in {\mathbb S}^{1}$. In particular, the elements
of ${\mathcal G_{\omega}}$ are, up to a constant complex phase,
real-valued and non-negative. Furthermore any real non-negative
ground state $u\in {\mathcal G_{\omega}}$ satisfies the following
properties
\begin{itemize}
\item[i)] \text{$u>0$ in $\R^N$},
\item[ii)]  \text{$u$ is a radially symmetric decreasing
  function with respect to some point},
\item[iii)]  $u \in C^2(\R^N)$,
\item[iv)] \text{for all $\alpha \in \N^N$ with $|\alpha| \leq 2$,
there exists $(c_{\alpha},\delta_{\alpha})\in(\R_+^*)^2$ such that}
$$
 |D^{\alpha}u(x)|
\leq C_{\alpha}e^{-\delta_{\alpha}|x|},\quad\text{for all $x\in\R^N$}.
$$
\end{itemize}
Moreover, in the case $N=1$, there exists a unique positive ground
state to \eqref{gs1} up to translations.
\end{theorem} \medskip

Observe that if $u \in {\mathcal G_{\omega}}$ is real and positive
any $v(x)=e^{i\theta}u(x-y)$ for $\theta\in {\mathbb S}^{1}$ and
$y \in \R^N$ belongs to ${\mathcal G_{\omega}}$. Except when $N=1$
we do not know if there exists a unique real positive ground state
up to translation. The proof of Theorem~\ref{ex} uses the
so-called dual approach introduced in~\cite{cojean} which
transforms equation~\eqref{gs1} into a semi-linear one which
belongs to the frame handle in~\cite{bere1,bere2}.
\medskip

Next we establish, for $p>1$ sufficiently large, a result of
instability by blow-up.
\begin{theorem}\label{instab}
    Assume that $\omega>0$,
    \begin{equation*}
        p>3+\frac{4}{N},
    \end{equation*}
    and that the conclusions of Theorem~\ref{cauchy} hold when
$\displaystyle f(s) = s^{\frac{p-1}{2}}$. Let
     $u \in X_{\C}$ be a ground state solution of
\begin{equation}
    \label{gseq}
    -\lapl u+u\lapl |u|^2+\omega u=|u|^{p-1}u\,\,\quad\text{in $\R^N$}.
\end{equation}
 Then, for all $\eps>0$, there exists
$a_0\in H^{s+2}(\R^N)$ such that $\|a_0-u\|_{H^1(\R^N)}<\eps$ and
the solution $\phi(t)$ of~\eqref{eq.schr1} with $\phi(0)=a_0$
blows up in finite time.
\end{theorem}

Observe that the assumptions of Theorems~\ref{cauchy}
and~\ref{instab} intersect for $p \geq 9$ when $N=1$, $p=7,9,11$
or $p \geq 13$ if $N=2$, $p=5,7,9$ if $N =3$ and $p=5$ if $N=4$.
\medskip

To prove Theorem~\ref{instab} we first establish a virial type
identity. Then, we introduce some sets which are invariant under
the flow, in the spirit of~\cite{BeCa}. At this point we take
advantage of ideas of~\cite{lecoz}. Namely, by introducing a
constrained approach and playing between various characterization
of the ground states, we are able to derive the blow up result
without having to solve directly a minimization problem, in
contrast to~\cite{BeCa}. \vskip3pt \noindent When $1<p <3 +
\frac{4}{N}$, we conjecture that the ground states solutions
of~\eqref{gs1} are orbitally stable. However, we did not manage to
prove this result (see Remark~\ref{lien} in that direction).
Instead, we consider the stability issue for the minimizers of the
problem
\begin{equation}
    \label{defvalc}
m(c)=\inf \{{\mathcal E}(u):\, \text{$u\in X$,\, $\|u\|_{2}^2 = c$}\},
\end{equation}
where the energy $\mathcal{E}$ reads as
$$
{\mathcal E}(u)=\frac{1}{2}\int_{\R^N}|\nabla u|^2dx+ \frac{1}{4}\int_{\R^N}
|\nabla |u|^2|^2 dx -\frac{1}{p+1}\int_{\R^N}|u|^{p+1}dx.
$$
We shall show that, if $p< 3 + \frac{4}{N}$ then $m(c)>-\infty$
for any $c
>0$. On the contrary, when $p>3 + \frac{4}{N}$, we have $m(c) = -
\infty$ for any $c>0$. See Lemma~\ref{boundedness}. \vskip4pt
\noindent Denote by $\mathcal{G}(c)$ the set of solutions to
\eqref{defvalc}. Our result of orbital stability is the following.

\begin{theorem}\label{stab}
Assume that
    \begin{equation*}
        1<p<3+\frac{4}{N},
    \end{equation*}
and let $c>0$ be such that $m(c)<0$. Then ${\mathcal G}(c)$ is non
void and  orbitally stable. Furthermore, in the two following
cases
\begin{align*}
&{\rm i)} \,\,\, 1<p<1+\frac{4}{N} \text{ and }c>0,\\
&{\rm ii)}\,\,\, 1+ \frac{4}{N} \leq p < 3 + \frac{4}{N} \text{
and } c
>0 \text{ is sufficiently large},
\end{align*}
we have $m(c)<0$.
\end{theorem}
Here we also have that, if $u \in {\mathcal G}(c)$, then any $v(x)
= e^{i \theta}u(x-y)$ for $\theta\in {\mathbb S}^{1}$ and $y \in
\R^N$ belongs to ${\mathcal G}(c)$. \vskip5pt \noindent In
Theorem~\ref{stab} when we say that ${\mathcal G}(c)$ is orbitally
stable we mean the following: For every $\eps>0$, there exists
$\delta>0$ such that, for any initial data $a_0\in X_{\C}\cap H^s(\R^N)$ such
that the conclusions of Theorem~\ref{cauchy} hold when $f(s) =
s^{\frac{p-1}{2}}$, if $\inf_{u\in{\mathcal
G}(c)}\|a_0-u\|_{H^1}<\delta$ then the solution $\phi(t,\cdot)$
of~\eqref{eq.schr1-intro} with initial condition $a_0$ satisfies
$$
\sup_{T_0\geq t\geq 0}\inf_{u\in {\mathcal
G}(c)}\|\phi(t,\cdot)-u\|_{H^1}<\eps,
$$
where $T_0$ is the existence time for $\phi$ given by Theorem
\ref{cauchy}. Note that our definition of stability requires the introduction
of the existence time $T_0$ for $\phi$ since we are not able to solve the Cauchy problem
\eqref{eq.schr1-intro} in the natural energy space $X_{\C}$.
The proof of Theorem~\ref{stab} is based on variational methods.
Assuming that $m(c) <0$ we obtain the convergence of any
minimizing sequence for~\eqref{defvalc} using concentration
compactness arguments and taking advantage of the autonomous
nature of~\eqref{defvalc}. This convergence result being
established, the proof of orbital stability follows in a standard fashion.
\medskip
\vskip3pt \noindent It is standard to show that $m(c)<0$ for all
$c>0$ when $1< p <1+ \frac{4}{N}$ (see~\cite{St}). When $1+
\frac{4}{N} \leq p < 3 + \frac{4}{N}$ we prove that there exists a
$c(p,N) >0$ such that $m(c) <0$ if $c > c(p,N)$. More precisely
the following occurs.

\begin{theorem}\label{no-minimizer}
Assume that $1 + \frac{4}{N} \leq p \leq 3 + \frac{4}{N}$. Then
there exists $c(p,N) >0$ such that
\begin{itemize}
\item[i)] If $0 < c < c(p,N)$ then  $m(c) =0$  and  $m(c)$  does not admit a minimizer.
\item[ii)] If $c > c(p,N)$ then $m(c)
<0$  and $m(c)$ admits a minimizer. In addition, 
the map $c \mapsto m(c)$ is strictly
decreasing.
\end{itemize}
\end{theorem}

\vskip3pt

As we have already mentioned, the paper~\cite{BoCh} has motivated
the present work. However, we stress that~\cite{BoCh} is only
partially correct and that it deals with orbital stability issues,
no results on the Cauchy problem nor of instability are presented.
In Remark~\ref{connections} we compare our results to the ones
of~\cite{BoCh}. Apart from~\cite{BoCh}, to the best of our
knowledge, we are not aware of any other results comparable to
those of our paper.

\vskip2pt
\noindent

\vskip22pt
\begin{center}\textbf{Notations.}\end{center}
\begin{enumerate}
\item For a function $f :\R^N\to \R^N$ and $1\leq j\leq N$, we denote by $\partial_j f$ the partial derivative with respect to the $j^{th}$ coordinate.
\item $M(\R^N)$ is the set of measurable functions in $\R^N$. For any
$p>1$ we denote by $L^p(\R^N)$ the space of $f$ in $M(\R^N)$
such that $\int_{\R^N}|f|^pdx<\infty$.
\item The norm $(\int_{\R^N}|f|^pdx)^{1/p}$ in $L^p(\R^N)$ is denoted by $\|\cdot\|_p$.
\item For $s\in\N$, we denote by $H^s(\R^N)$ the Sobolev space of functions $f$ in $L^2(\R^N)$
having generalized partial derivatives $\partial_i^k f$ in
$L^2(\R^N)$, for $i=1,\dots, N$ and $0\leq k \leq s$.
\item The norm $(\int_{\R^N}|f|^2dx+ \int_{\R^N}|\nabla f|^2dx)^{1/2}$ in $H^1(\R^N)$ is denoted by $\|\cdot\|$ and more generally, the norm in $H^s$ is denoted by $\| \cdot\|_{H^s}$.
\item ${\mathcal L}^N(E)$ denotes the Lebesgue measure of a measurable set $E\subset\R^N$.
\item For $R>0$, $B(0,R)$ is the ball in $\R^N$ centered at zero with radius $R$.
\item $\Re(z)$ (resp.\ $\Im(z)$) denotes the real part (resp.\ the imaginary part) of a complex number $z$.
\item For a real number $r$, we denote by $E(r)$ the integer part of $r$.
\item $X$ denotes the restriction of $X_{\C}$ to real functions.
\item $K, K(p,N)$ denote various constants which are not
essential in the problem and may vary from line to line.
\end{enumerate}
\bigskip
\begin{center}\textbf{Organization of the paper.}\end{center}
In Section~\ref{cauchyproblem}, we prove Theorem~\ref{cauchy}
concerning the well-posedness result for
equation~\eqref{eq.schr1-intro}. In Section~\ref{instability}, we
establish the existence and properties of the ground states
solutions of~\eqref{gs1}, Theorem~\ref{ex} and we prove the
instability result, Theorem~\ref{instab}. In
Section~\ref{stationary-problems2}, we study the minimization
problem~\eqref{defvalc}. Assuming that $m(c) <0$  we prove the
existence of a minimizer and we study under which conditions $m(c)
<0$ holds. Finally, in Section~\ref{stability}, we prove the
convergence of all the minimizing sequences of~\eqref{defvalc} and
thus derive the stability result, Theorem~\ref{stab}.

\bigskip
\noindent {\bf Acknowledgments.} The first author thanks M. Ohta for helpful
observations concerning the Cauchy problem. The second author thanks Patrick
Hild for stimulating discussions around the non-existence result of
Theorem~\ref{no-minimizer}.

\medskip

\section{The Cauchy problem}\label{cauchyproblem}

This section is fully devoted to the proof of Theorem
\ref{cauchy}. \medskip

We first rewrite equation~\eqref{eq.schr1-intro} into a system
involving $\phi$ and $\overline{\phi}$ in the following way
\begin{eqnarray}\label{cauchy2}
2i \left( \begin{array}{l}
    \phi_t\\\overline{\phi}_t\end{array} \right)
+\mathcal{A}(\phi)\left( \begin{array}{l} \Delta \phi\\ \Delta
\overline{\phi} \end{array} \right) -\left(\begin{array}{l} 2\phi
|\nabla
    \phi|^2+\phi f(|\phi|^2)\\
    -2\overline{\phi}|\nabla
    \phi|^2-\overline{\phi}f(|\phi|^2)\end{array}\right)=0,
\end{eqnarray}
where $$\mathcal{A}(\phi)=\left(\begin{array}{ll}
1+|\phi|^2&\hspace{8mm}\phi^2\\
-\overline{\phi}^2&-(1+|\phi|^2)\end{array}\right).$$ A direct
calculation shows that $\mathcal{A}(\phi)$ is invertible and that
$$\mathcal{A}^{-1}(\phi)=\frac{1}{1+2|\phi|^2}\mathcal{A}(\phi).$$
In order to overcome the loss of derivatives and to linearize the
quadratic term involving $\nabla \phi$, we differentiate the
equation with respect to space and time variables to obtain a new
system in $\phi_0,\dots,\phi_{N+2}$ where $\phi_0=\phi$ and
\begin{align}\label{new_unkn}
\forall 1\leq j\leq N, \hspace{2mm} \phi_j=\partial_j \phi,
\hspace{2mm} \phi_{N+1}=e^{g(|\phi|^2)}\phi_t, \hspace{2mm}
\phi_{N+2}=e^{q(|\phi|^2)}\Delta \phi.
\end{align}
The functions $g$ and $q$ are used as gauge transforms and their
role will be explain later. We also set $\Phi^*=(\phi_j)_{j=0}^N$
and $\Phi=(\phi_j)_{j=0}^{N+2}$. Equation~\eqref{cauchy2} can be
rewritten as
\begin{eqnarray}\label{cauchy3}
2i \left( \begin{array}{l}
    (\phi_0)_t\\(\overline{\phi}_0)_t\end{array} \right)
+\mathcal{A}(\phi_0)\left( \begin{array}{l} \Delta \phi_0\\ \Delta
\overline{\phi_0} \end{array} \right) +\mathcal{F}_0(\Phi^*)=0,
\end{eqnarray}
where $\mathcal{F}_0$ is a smooth function depending only on
$\Phi^*$. Differentiating equation~\eqref{cauchy3} with respect to
$x_j$ for $j=1,\dots,N$, we obtain
\begin{align*}
&2i \left( \begin{array}{l}
    (\phi_j)_t\\(\overline{\phi}_j)_t\end{array} \right)
+\mathcal{A}(\phi_0)\left( \begin{array}{l} \Delta \phi_j\\ \Delta
    \overline{\phi}_j \end{array} \right)
 +\sum_{k=1}^N B(\phi_0,\phi_k)
\left( \begin{array}{l} T_{kj} \phi_{N+2}  \\ T_{kj}
\overline{\phi}_{N+2}
  \end{array}\right)\\
&+C(\phi_0,\phi_j)\left( \begin{array}{l} e^{-q(|\phi_0|^2)}
\phi_{N+2} \\ e^{-q(|\phi_0|^2)}
    \overline{\phi}_{N+2} \end{array} \right)
+\left(
  \begin{array}{l}F(\Phi^*,\phi_j)\\-\overline{F}(\Phi^*,\phi_j)
  \end{array}\right) =0,
\end{align*}
where $B$, $C$ and $F$ are smooth functions of their arguments and
especially
$$C(\phi_0,\phi_j)= \partial_j \mathcal{A}(\phi_0)=
\left( \begin{array}{ll}
\phi_0\overline{\phi_j}+\overline{\phi_0}\phi_j& 2\phi_0\phi_j\\
-2\overline{\phi}_0\overline{\phi}_j&-\phi_0\overline{\phi_j}-\overline{\phi_0}\phi_j
\end{array}\right).
$$

For $i,j=1,\dots,N$, $T_{ij}$ is the following operator of order 0
$$T_{ij} \phi=
\partial_i \partial_j \Delta^{-1} (e^{-q(|u_0|^2)}\phi).$$
We can rewrite these equations as follows
\begin{eqnarray}\label{cauchy4}
2i \left( \begin{array}{l}
    (\phi_j)_t\\(\overline{\phi}_j)_t\end{array} \right)
+\mathcal{A}(\phi_0)\left( \begin{array}{l} \Delta \phi_j\\ \Delta
    \overline{\phi}_j \end{array} \right)
 +\mathcal{F}_j(\Phi^*,\phi_{N+2},T \phi_{N+2})=0,
\end{eqnarray}
where $\mathcal{F}_j$ is a smooth function of its arguments.
Differentiating equation~\eqref{cauchy3} with respect to $t$, we
derive
\begin{eqnarray}\label{cauchy5}
\begin{array}{l}
\displaystyle 2i \left( \begin{array}{l}
    (e^{-f(|\phi_0|^2)}\phi_{N+1})_t\\(e^{-f(|\phi_0|^2)}\overline{\phi}_{N+1})_t\end{array} \right)
+C(\phi_0,e^{-f(|\phi_0|^2)}\phi_{N+1}) \left(
  \begin{array}{l} e^{-q(|\phi_0|^2)}\phi_{N+2}\\e^{-q(|\phi_0|^2)} \overline{\phi}_{N+2} \end{array}\right)\\
\displaystyle +\mathcal{A}(\phi_0)\left( \begin{array}{l} \Delta
(e^{-f(|\phi_0|^2)}\phi_{N+1})\\ \Delta
   (e^{-f(|\phi_0|^2)} \overline{\phi}_{N+1} )\end{array} \right)
+\sum_{k=1}^N B(\phi_0,\phi_k) \left( \begin{array}{l} \partial_k
    (e^{-f(|\phi_0|^2)} \phi_{N+1}) \\ \partial_k (e^{-f(|\phi_0|^2)}
    \overline{\phi}_{N+1})\end{array} \right) \\
\displaystyle+\left(\begin{array}{l}
F(\Phi^*,e^{-f(|\phi_0|^2)}\phi_{N+1})\\-\overline{F}(\Phi^*,e^{-f(|\phi_0|^2)}\phi_{N+1})\end{array}\right)=0,
\end{array}
\end{eqnarray}
which can be rewritten as
\begin{eqnarray}\label{cauchy6}
\begin{array}{l}
\displaystyle 2i \left(
\begin{array}{l}(\phi_{N+1})_t\\(\overline{\phi}_{N+1})_t
\end{array}\right)
+\mathcal{A}(\phi_0)\left( \begin{array}{l} \Delta \phi_{N+1} \\
\Delta
    \overline{\phi}_{N+1} \end{array} \right) +\sum_{k=1}^N
\mathcal{D}(\phi_0,\phi_k) \left( \begin{array}{l} \partial_k \phi_{N+1} \\
    \partial_k \overline{\phi}_{N+1} \end{array} \right)\\
\displaystyle  + \mathcal{G}(\Phi,T\phi_{N+2})=0,
\end{array}
\end{eqnarray}
where $\mathcal{D}$ and $\mathcal{G}$ are smooth functions of
their arguments. By applying the operator $\Delta$ on
equation~\eqref{cauchy3}, we obtain
\begin{eqnarray} \label{cauchy7}
\begin{array}{l}
\displaystyle 2i \left(\begin{array}{l} (\phi_{N+2})_t \\
(\overline{\phi}_{N+2})_t \end{array}\right)+ \mathcal{A}(\phi_0)
\left( \begin{array}{l} \Delta \phi_{N+2} \\ \Delta
    \overline{\phi}_{N+2} \end{array} \right) +\sum_{k=1}^N
\mathcal{E}(\phi_0,\phi_k) \left( \begin{array}{l} \partial_k \phi_{N+2} \\
    \partial_k \overline{\phi}_{N+2} \end{array} \right)\\
\displaystyle +\mathcal{I}(\Phi,T\phi_{N+2})=0,
\end{array}
\end{eqnarray}
where $\mathcal{E}$ and $\mathcal{I}$ are also smooth functions of
their arguments. At this point, we need to make more precise the
matrices $B$, $\mathcal{D}$ and $\mathcal{E}$ since they represent
the quasi-linear part of the equations. A direct computation gives
$$B(\phi_0,\phi_k)=\left(
  \begin{array}{ll}
  2\phi_0 \overline{\phi}_k & 2\phi_0\phi_k\\
  -2\overline{\phi}_0\overline{\phi}_k&-2\overline{\phi}_0\phi_k
\end{array}\right),$$
$$\mathcal{D}(\phi_0,\phi_k)=B(\phi_0,\phi_k)-2f^{'}(|\phi_0|^2)\mathcal{A}(\phi_0)\left(
  \begin{array}{ll}\phi_0 \overline{\phi}_k+\overline{\phi}_0 \phi_k
&\hspace{9mm}0\\ \hspace{9mm}0 & \phi_0 \overline{\phi}_k
+\overline{\phi}_0 \phi_k\end{array}\right),$$
\begin{align*}
\mathcal{E}(\phi_0,\phi_k)=& B(\phi_0,\phi_k) +2C(\phi_0,\phi_k)
\\&-2 \mathcal{A}(\phi_0)q^{'}(|\phi_0|^2)\left(
  \begin{array}{ll} \phi_0 \overline{\phi}_k +\overline{\phi}_0 \phi_k
    &\hspace{9mm}0\\\hspace{9mm}0&\phi_0\overline{\phi}_k +\overline{\phi}_0 \phi_k
\end{array}\right).
\end{align*}
Usual energy estimates for Schr\"odinger equations requires that
the diagonal coefficients of $\mathcal{D}$ and $\mathcal{E}$ in
equations~\eqref{cauchy6} and~\eqref{cauchy7} are purely
imaginary. Roughly speaking, this allows to integrate by parts the
bad terms including first order derivatives of the unknown. This
is why we make use of gauge transforms $g$ and $q$. Finally,
in order to avoid any smallness assumption on the initial data, we
need to transform slightly equation~\eqref{cauchy3} in the
following way. We multiply the equation by
$\mathcal{A}^{-1}(\phi_0)$ and we split the matrix in front of the
time derivatives of $\phi_0$ into
$$\mathcal{A}^{-1}(\phi_0)=Id +\Big(\mathcal{A}^{-1}(\phi_0)-Id\Big),$$
where $Id$ is the $2 \times 2$ identity matrix. Then recalling
that $\partial_t \phi_0=e^{-g(|\phi_0|^2)}\phi_{N+1}$, we rewrite
equation \eqref{cauchy3} in
\begin{eqnarray}\label{cauchy7.5}
2i \left( \begin{array}{l}
    (\phi_0)_t\\(\overline{\phi}_0)_t\end{array} \right)
+\left( \begin{array}{l} \Delta \phi_0\\ \Delta \overline{\phi_0}
\end{array} \right) +\mathcal{G}_0(\Phi)=0,
\end{eqnarray}
where
$$\mathcal{G}_0(\Phi)=\mathcal{A}^{-1}(\phi_0)\mathcal{F}_0(\Phi^*) +ie^{-g(|\phi_0|^2)}\Big(\mathcal{A}^{-1}(\phi_0)-Id\Big)
\left(\begin{array}{l} \phi_{N+1}\\\overline{\phi}_{N+1}
\end{array}\right).$$
We then have transformed equation~\eqref{eq.schr1-intro} into the
following system
\begin{eqnarray}\label{cauchy8}
2i \left( \begin{array}{l}
    (\phi_0)_t\\(\overline{\phi}_0)_t\end{array} \right)
+\left( \begin{array}{l} \Delta \phi_0\\ \Delta \overline{\phi_0}
\end{array} \right) +\mathcal{G}_0(\Phi)=0,
\end{eqnarray}
for $j=1,\dots,N$
\begin{eqnarray}\label{cauchy9}
2i \left( \begin{array}{l}
    (\phi_j)_t\\(\overline{\phi}_j)_t\end{array} \right)
+\mathcal{A}(\phi_0)\left( \begin{array}{l} \Delta \phi_j\\ \Delta
    \overline{\phi}_j \end{array} \right)
 +\mathcal{F}_j(\Phi^*,\phi_{N+2},T \phi_{N+2})=0,
\end{eqnarray}
\begin{eqnarray}\label{cauchy10}
\begin{array}{l}
\displaystyle 2i \left(
\begin{array}{l}(\phi_{N+1})_t\\(\overline{\phi}_{N+1})_t
\end{array}\right)
+\mathcal{A}(\phi_0)\left( \begin{array}{l} \Delta \phi_{N+1} \\
\Delta
    \overline{\phi}_{N+1} \end{array} \right) +\sum_{k=1}^N
\mathcal{D}(\phi_0,\phi_k) \left( \begin{array}{l} \partial_k \phi_{N+1} \\
    \partial_k \overline{\phi}_{N+1} \end{array} \right)\\
\displaystyle  + \mathcal{G}(\Phi,T\phi_{N+2})=0,
\end{array}
\end{eqnarray}
\begin{eqnarray} \label{cauchy11}
\begin{array}{l}
\displaystyle 2i \left(\begin{array}{l} (\phi_{N+2})_t \\
(\overline{\phi}_{N+2})_t \end{array}\right)+ \mathcal{A}(\phi_0)
\left( \begin{array}{l} \Delta \phi_{N+2} \\ \Delta
    \overline{\phi}_{N+2} \end{array} \right) +\sum_{k=1}^N
\mathcal{E}(\phi_0,\phi_k) \left( \begin{array}{l} \partial_k \phi_{N+2} \\
    \partial_k \overline{\phi}_{N+2} \end{array} \right)\\
\displaystyle +\mathcal{I}(\Phi,T\phi_{N+2})=0.
\end{array}
\end{eqnarray}
We now apply a fixed point theorem to system
\eqref{cauchy8}-\eqref{cauchy11}. Let $s$ be as in
Theorem~\ref{cauchy} and introduce the function space
$$\mathcal{X}_T= \left \{\begin{array}{l}  \Phi=(\phi_j)_{j=0}^{N+2}: \phi_j \in C([0,T];L^2(\R^N))
\cap L^{\infty}(0,T;H^{s}(\R^N)),  \\
\|\Phi\|_{X_T}=\sum_{j=0}^{N+2} \sup_{0 \leq t \leq T}
\|\phi_j(t)\|_{H^s(\R^N)} < \infty \end{array}\right\}.$$ For
$M=(m_j)_{j=0}^{N+2}\in (\R^*_+)^{N+3}$ and  $r\in \R_+^*$, we
denote
$$
\mathcal{X}_T(M,r)= \left\{ \begin{array}{l}
\Phi=(\phi_j)_{j=0}^{N+2}\in \mathcal{X}_T
: \forall j=0,..,N+2 \hspace{2mm} \|\phi_j\|_{L^{\infty}(0,T;H^s(\R^N))} \leq m_j \\
\|(\phi_0)_t\|_{L^{\infty}(0,T;H^{E(\frac{N}{2})+1}(\R^N))} \leq r
\text{ and } \phi_0(0,x)=a_0(x)
\end{array} \right\},$$
and let $\Psi=(\psi_j)_{j=0}^{N+2}\in \mathcal{X}_T(M,r)$. Denote
$\Psi^*=(\psi_j)_{j=0}^{N}$ and consider the linearized version of
system~\eqref{cauchy8}-\eqref{cauchy11} as follows
\begin{eqnarray}\label{cauchy12}
2i \left( \begin{array}{l}
    (\phi_0)_t\\(\overline{\phi}_0)_t\end{array} \right)
+\left( \begin{array}{l} \Delta \phi_0\\ \Delta \overline{\phi_0}
\end{array} \right) +\mathcal{G}_0(\Psi)=0,
\end{eqnarray}
for $j=1,\dots,N$
\begin{eqnarray}\label{cauchy13}
2i \left( \begin{array}{l}
    (\phi_j)_t\\(\overline{\phi}_j)_t\end{array} \right)
+\mathcal{A}(\psi_0)\left( \begin{array}{l} \Delta \phi_j\\ \Delta
    \overline{\phi}_j \end{array} \right)
 +\mathcal{F}_j(\Psi^*,\psi_{N+2},T \psi_{N+2})=0,
\end{eqnarray}
\begin{eqnarray}\label{cauchy14}
\begin{array}{l}
\displaystyle 2i \left(
\begin{array}{l}(\phi_{N+1})_t\\(\overline{\phi}_{N+1})_t
\end{array}\right)
+\mathcal{A}(\psi_0)\left( \begin{array}{l} \Delta \phi_{N+1} \\
\Delta
    \overline{\phi}_{N+1} \end{array} \right) +\sum_{k=1}^N
\mathcal{D}(\psi_0,\psi_k) \left( \begin{array}{l} \partial_k \phi_{N+1} \\
    \partial_k \overline{\phi}_{N+1} \end{array} \right)\\
\displaystyle  + \mathcal{G}(\Psi,T\psi_{N+2})=0,
\end{array}
\end{eqnarray}
\begin{eqnarray} \label{cauchy15}
\begin{array}{l}
\displaystyle 2i \left(\begin{array}{l} (\phi_{N+2})_t \\
(\overline{\phi}_{N+2})_t \end{array}\right)+ \mathcal{A}(\psi_0)
\left( \begin{array}{l} \Delta \phi_{N+2} \\ \Delta
    \overline{\phi}_{N+2} \end{array} \right) +\sum_{k=1}^N
\mathcal{E}(\psi_0,\psi_k) \left( \begin{array}{l} \partial_k \phi_{N+2} \\
    \partial_k \overline{\phi}_{N+2} \end{array} \right)\\
\displaystyle +\mathcal{I}(\Psi,T\psi_{N+2})=0.
\end{array}
\end{eqnarray}
Let $\mathcal{Z}=\left[ L^{\infty}(0,T;H^s(\R^N)) \cap
 C([0,T];L^2(\R^N)) \right]^{N+3}$.
Then the Cauchy problem $\eqref{cauchy12}$-$\eqref{cauchy15}$ with
initial condition
\begin{align*}
&\phi_0(0,x)=a_0(x), \; \text{for } j=1,\dots,N, \; \phi_j(0,x)=\partial_j a_0(x),\\
&\phi_{N+1}(0,x)=\frac{1}{2i}e^{g(|a_0(x)|^2)}\left(-\mathcal{A}(\phi_0(0))\Delta a_0(x) -\mathcal{F}_0(\Psi^*(0))\right),\\
&\phi_{N+2}(0,x)=e^{q(|a_0(x)|^2)}\Delta a_0(x),
\end{align*}
defines a mapping $\mathcal{S}$
\begin{align*}
\mathcal{S}: \hspace{2mm}&\mathcal{Z}\longrightarrow \mathcal{Z}\\
&\Psi \longmapsto \Phi.
\end{align*}
For more details on the existence result for
system~\eqref{cauchy12}-\eqref{cauchy15}, we refer to~\cite{co}
and~\cite{poppenberg}. In order to prove Theorem~\ref{cauchy}, we
have to find a time $T>0$ and constants $M\in (\R_+^*)^{N+3}$ and
$r\in \R_+^*$ such that $\mathcal{S}$ maps  the closed ball
${\mathcal X}_T(M,r)$ into itself and is a contraction mapping
under the constraint that it acts on ${\mathcal X}_T(M,r)$ in the
norm $\sum_{j=0}^{N+2}\text{sup}_{t\in[0,T]}\|\phi_j \|_{L^2}$. We
begin with equation \eqref{cauchy15} and perform an
$H^s$-estimate. Following \cite{co}, we apply the operator
$(1-\Delta)^{\frac{s}{2}}$ on equation \eqref{cauchy15} and
multiply the resulting equation by $\mathcal{A}^{-1}(\phi_0)$ to
obtain, denoting $\chi= (1-\Delta)^{\frac{s}{2}}\phi_{N+2}$
\begin{eqnarray} \label{cauchy16}
\begin{array}{l}
\displaystyle 2i \mathcal{A}^{-1}(\psi_0) \left(\begin{array}{l}
(\chi)_t \\ (\overline{\chi})_t \end{array}\right)+
 \left( \begin{array}{l} \Delta \chi \\ \Delta
    \overline{\chi} \end{array} \right) +\sum_{k=1}^N
\mathcal{L}(\psi_0,\psi_k,\partial_k\psi_0) \left( \begin{array}{l} \partial_k \chi \\
    \partial_k \overline{\chi}_{N+2} \end{array} \right)\\
\displaystyle
+\mathcal{J}_{j=0}^{s}(D^j\Psi,D^j\phi_{N+2},T\psi_{N+2})=0
\end{array}
\end{eqnarray}
where $D^j$ denotes any space derivation of order less or equal to
$s$ with respect to the $j^{th}$ space coordinate. The matrix $\mathcal{L}$ reads
$$\mathcal{L}(\psi_0,\psi_k,\partial_k\psi_0)=
\mathcal{A}^{-1}(\psi_0)\Big(\mathcal{E}(\psi_0,\psi_k)+s\partial_k
\mathcal{A}(\psi_0)\Big).$$ We notice here that the dependence of
$\mathcal{J}$ in $\phi_{N+2}$ and its derivatives  is affine. We
are now able to choose the gauge transform $q$. Recall that
\begin{align*}
\mathcal{E}(\psi_0,\psi_k)=& B(\psi_0,\psi_k) +2C(\psi_0,\psi_k)
\\&-2 \mathcal{A}(\psi_0)q^{'}(|\psi_0|^2)\left(
  \begin{array}{ll} \psi_0 \overline{\psi}_k +\overline{\psi}_0 \psi_k
    &\hspace{9mm}0\\\hspace{9mm}0&\psi_0\overline{\psi}_k +\overline{\psi}_0 \psi_k
\end{array}\right),
\end{align*}
a direct calculation shows that for $j=1,2$ (denoting by $b^{11}$
and $b^{22}$ the diagonal coefficients of a 2x2 matrix $b$),
$$\Re\Big(\mathcal{A}^{-1}(\psi_0)\big(B(\psi_0,\psi_k) +2C(\psi_0,\psi_k)\big)\Big)^{jj}=
\frac{3}{1+2|\psi|^2}\big( \psi_0\overline{\psi}_k
+\overline{\psi}_0 \psi_k\big).
$$
Then choosing
$$q(\sigma)=\frac{3}{4}\text{ln}(1+2\sigma)$$
gives
$$\Re\Big(\mathcal{A}^{-1}(\psi_0) \mathcal{E}(\psi_0,\psi_k)\Big)^{jj}=0.$$
Furthermore, by differentiating equation \eqref{cauchy16} $s$
times in space, we add in matrix $\mathcal{L}$ the term
$s\mathcal{A}^{-1}(\psi_0)\partial_k \mathcal{A}(\psi_0)$ which is
not eliminated by $q$. As a consequence we have to use a second
gauge transform by putting $\kappa=e^{b(|\psi_0|^2)}\chi$ solution
to
\begin{eqnarray} \label{cauchy17}
\begin{array}{l}
\displaystyle 2i \mathcal{A}^{-1}(\psi_0) \left(\begin{array}{l}
(\kappa)_t \\ (\overline{\kappa})_t \end{array}\right)+
 \left( \begin{array}{l} \Delta \kappa \\ \Delta
    \overline{\kappa} \end{array} \right) +\sum_{k=1}^N
\mathcal{M}(\psi_0,\psi_k,\partial_k\psi_0) \left( \begin{array}{l} \partial_k \kappa \\
    \partial_k \overline{\kappa}_{N+2} \end{array} \right)\\
\displaystyle
+\mathcal{K}_{j=0}^{s}(D^j\Psi,D^j\phi_{N+2},T\psi_{N+2},(\psi_0)_t)=0,
\end{array}
\end{eqnarray}
where
$$\mathcal{M}(\psi_0,\psi_k,\partial_k \psi_0)=
\mathcal{L}(\psi_0,\psi_k,\partial_k \psi_0) -2
\left(\begin{array}{ll}\partial_k b(|\psi_0|^2) &\hspace{4mm}0 \\
\hspace{4mm}0&
\partial_k b(|\psi_0|^2) \end{array}\right).
$$
Note that  the matrix $\mathcal{K}$ depends also on $(\psi_0)_t$.
Once again, an easy calculation shows that if we choose $b$ such
that
$$b(\sigma)=\frac{s}{4} \text{ln}(1+2\sigma),$$
then for $j=1,2$
$$\Re\Big(s\mathcal{A}^{-1}(\psi_0)\partial_k\mathcal{A}(\psi_0)-2\left(\begin{array}{ll}\partial_k b(|\psi_0|^2) &\hspace{4mm}0 \\ \hspace{4mm}0&
\partial_k b(|\psi_0|^2) \end{array}\right)\Big)^{jj}=0.$$
We are now able to perform the suitable energy estimate on
equation~\eqref{cauchy17}. Multiplying equation \eqref{cauchy17}
by $\overline{\kappa}$, integrate over $\R^N$ and taking the first
line of the resulting system leads to
\begin{align}\label{cauchy18}
&i\int_{\R^N}\frac{1+|\psi_0|^2}{1+2|\psi_0|^2}\kappa_t\overline{\kappa}dx+i\int_{\R^N}
\frac{\psi_0^2}{1+2|\psi_0|^2}\overline{\kappa}_t \overline{\kappa} dx+\int_{\R^N}\Delta \kappa \overline{\kappa}dx\nonumber\\
&+\sum_{k=1}^N\int_{\R^N} \mathcal{M}^{11}(\psi_0,\psi_k,\partial_k\psi_0)(\partial_k \kappa) \overline{\kappa}dx\nonumber\\
&+\int_{\R^N}\mathcal{M}^{12}(\psi_0,\psi_k,\partial_k\psi_0)(\partial_k \overline{\kappa} )\overline{\kappa}dx\\
&+\int_{\R^N}\mathcal{K}_{j=0}^{s}(D^j\Psi,D^j\phi_{N+2},T\psi_{N+2},(\psi_0)_t)\overline{\kappa}dx.\nonumber
\end{align}
We take the imaginary part of equation~\eqref{cauchy18}. We have
\begin{align*}
&\Im\Big(
i\int_{\R^N}\frac{1+|\psi_0|^2}{1+2|\psi_0|^2}\kappa_t\overline{\kappa}dx+i\int_{\R^N}
\frac{\psi_0^2}{1+2|\psi_0|^2}\overline{\kappa}_t \overline{\kappa} dx\Big)\\
&=\int_{\R^N}\frac{1+|\psi_0|^2}{2+4|\psi_0|^2}|\kappa|_t^2dx+\int_{\R^N}
\Big(\frac{\psi_0^2}{4(1+2|\psi_0|^2)}(\overline{\kappa}^2)_t+\frac{\overline{\psi}_0^2}{4(1+2|\psi_0|^2)}(\kappa^2)_t\Big)dx\\
&=\frac{d}{dt}\Big(
\int_{\R^N}\frac{1+|\psi_0|^2}{2+4|\psi_0|^2}|\kappa|^2dx+\int_{\R^N}
\Big(\frac{\psi_0^2}{4(1+2|\psi_0|^2)}\overline{\kappa}^2+\frac{\overline{\psi}_0^2}{4(1+2|\psi_0|^2)}\kappa^2\Big)dx\Big)\\
&-
\int_{\R^N}\Big(\frac{1+|\psi_0|^2}{2+4|\psi_0|^2}\Big)_t|\kappa|^2dx-\int_{\R^N}
\Big(\Big(\frac{\psi_0^2}{4(1+2|\psi_0|^2)}\Big)_t\overline{\kappa}^2+\Big(\frac{\overline{\psi}_0^2}{4(1+2|\psi_0|^2)}\Big)_t\kappa^2\Big)dx.
\end{align*}
The other terms in equation \eqref{cauchy18} are classical and can
be treated exactly as in \cite{co}. The important point to notice
is that since the diagonal coefficients of $\mathcal{M}$ are pure
imaginary, one has for $k=1,\dots,N$
\begin{align*}
\Im\Big(\int_{\R^N} \mathcal{M}^{11}(\psi_0,\psi_k,\partial_k\psi_0)(\partial_k \kappa) \overline{\kappa}dx\Big)&=\int_{\R^N} \im\big(\mathcal{M}^{11}(\psi_0,\psi_k,\partial_k\psi_0)\big)\partial_k \frac{|\kappa|^2}{2}dx,\\
&=-\int_{\R^N}\partial_k\Big(\im\big(
\mathcal{M}^{11}(\psi_0,\psi_k,\partial_k\psi_0)\big)\Big)\frac{|\kappa|^2}{2}
dx,
\end{align*}
by integration by parts. This allows to overcome the loss of
derivatives of this quasi-linear Schr\"odinger equation and brings
the following estimate
\begin{align}\label{cauchy19}
&\frac{d}{dt}\Big(
\int_{\R^N}\frac{1+|\psi_0|^2}{2+4|\psi_0|^2}|\kappa|^2dx+\int_{\R^N}
\Big(\frac{\psi_0^2}{4(1+2|\psi_0|^2)}\overline{\kappa}^2+\frac{\overline{\psi}_0^2}{4(1+2|\psi_0|^2)}\kappa^2\Big)dx\Big)\nonumber\\
&\leq 4\int_{\R^N}
(|\psi_0|^2)_t|\kappa|^2dx+C_1(M,r)\|\kappa\|_2^2,
\end{align}
where $C_1(M,r)$ is a constant depending only on $M$ and $r$. To
derive inequality \eqref{cauchy19}, we have used the fact that
\begin{align*}
&\Big(\frac{1+|\psi_0|^2}{2+4|\psi_0|^2}\Big)_t=
\frac{(|\psi_0|^2)_t}{2+4|\psi_0|^2}-4\Big(\frac{1+|\psi_0|^2}{(2+4|\psi_0|^2)^2}\Big)(|\psi_0|^2)_t
\end{align*}
\begin{align*}
&\Big(\frac{\psi_0^2}{4(1+2|\psi_0|^2)}\Big)_t=
\frac{(\psi_0^2)_t}{4(1+2|\psi_0|^2)}-\Big(\frac{\psi_0^2}{2(1+2|\psi_0|^2)^2}\Big)(|\psi_0|^2)_t
\end{align*}
\begin{align*}
&\Big(\frac{\overline{\psi}_0^2}{4(1+2|\psi_0|^2)}\Big)_t=
\frac{(\overline{\psi}_0^2)_t}{4(1+2|\psi_0|^2)}-\Big(\frac{\overline{\psi}_0^2}{2(1+2|\psi_0|^2)^2}\Big)(|\psi_0|^2)_t
\end{align*}
and then
\begin{align*}
&\Big|
\int_{\R^N}\Big(\frac{1+|\psi_0|^2}{2+4|\psi_0|^2}\Big)_t|\kappa|^2dx-\int_{\R^N}
\Big(\Big(\frac{\psi_0^2}{4(1+2|\psi_0|^2)}\Big)_t\overline{\kappa}^2+\Big(\frac{\overline{\psi}_0^2}{4(1+2|\psi_0|^2)}\Big)_t\kappa^2\Big)dx\Big|\\
&\leq 4  \int_{\R^N} (|\psi_0|^2)_t|\kappa|^2dx.
\end{align*}
Using the fact that
$$
\sup_{t\in[0,T]}\|(\psi_0)_t\|_{H^{E(\frac{N}{2})+1}(\R^N)}\leq r
$$
and the continuous embedding
$H^{E(\frac{N}{2})+1}(\R^N)\hookrightarrow L^\infty(\R^N)$, we can
find a constant $C_2(M,r)$ such that
\begin{align}\label{cauchy20}
&\displaystyle\frac{d}{dt}\Big(
\int_{\R^N}\frac{1+|\psi_0|^2}{2+4|\psi_0|^2}|\kappa|^2dx+\int_{\R^N}
\Big(\frac{\psi_0^2}{4(1+2|\psi_0|^2)}\overline{\kappa}^2+\frac{\overline{\psi}_0^2}{4(1+2|\psi_0|^2)}\kappa^2\Big)dx\Big)\nonumber\\
&\leq C_2(M,r)\|\kappa\|_2^2.
\end{align}
Integrating inequality~\eqref{cauchy20} from 0 to $t$ gives
\begin{align*}
&\int_{\R^N}\frac{1+|\psi_0(t)|^2}{2+4|\psi_0(t)|^2}|\kappa(t)|^2dx+\int_{\R^N}
\Big(\frac{\psi_0^2(t)}{4(1+2|\psi_0(t)|^2)}\overline{\kappa}^2(t)+\frac{\overline{\psi}_0^2(t)}{4(1+2|\psi_0(t)|^2)}\kappa^2(t)\Big)dx\\
&\leq
\int_{\R^N}\frac{1+|\psi_0(0)|^2}{2+4|\psi_0(0)|^2}|\kappa(0)|^2dx+\int_{\R^N}
\Big(\frac{\psi_0^2(0)}{4(1+2|\psi_0(0)|^2)}\overline{\kappa}^2(0)+\frac{\overline{\psi}_0^2(0)}{4(1+2|\psi_0(0)|^2)}\kappa^2(0)\Big)dx\\
&+C_2(M,r) \int_0^t \|\kappa(s)\|_2^2ds.
\end{align*}
For all $t\in [0,T]$, we have
$$\frac{1+|\psi_0(t)|^2}{2+4|\psi_0(t)|^2}|\kappa(t)|^2+\frac{\psi_0^2(t)}{4(1+2|\psi_0(t)|^2)}\overline{\kappa}^2(t)+\frac{\overline{\psi}_0^2(t)}{4(1+2|\psi_0(t)|^2)}\kappa^2(t)\geq \frac{1}{2+4|\psi_0(t)|^2}|\kappa(t) |^2.$$
Denoting by
$$CI_{N+2}(0)= \int_{\R^N}\frac{1+|\psi_0(0)|^2}{2+4|\psi_0(0)|^2}|\kappa(0)|^2dx+\int_{\R^N}
\Big(\frac{\psi_0^2(0)}{4(1+2|\psi_0(0)|^2)}\overline{\kappa}^2(0)+\frac{\overline{\psi}_0^2(0)}{4(1+2|\psi_0(0)|^2)}\kappa^2(0)\Big)dx,$$
we derive
\begin{align}\label{cauchy21}
\int_{\R^N}\frac{1}{2+4|\psi_0(t)|^2}|\kappa (t)|^2dx \leq
CI_{N+2}(0)+ C_2(M,r) \int_0^t \|\kappa(s)\|_2^2ds.
\end{align}
Recalling that $\psi_0\in L^\infty(0,T;H^{s}(\R^N))$ and the
continuous embedding $H^{s}(\R^N)\hookrightarrow L^\infty(\R^N)$
and denote by $C_b$ the best constant of this embedding, we have
$$\|\psi_0(t)\|_{L^\infty(\R^N)}\leq C_b m_0.$$
This provides
$$\int_{\R^N}\frac{1}{2+4|\psi_0(t)|^2}|\kappa(t) |^2dx \geq \frac{1}{2+4C_b^2m_0^2}\int_{\R^N}|\kappa(t) |^2dx$$
which gives
\begin{align}\label{cauchy22}
\int_{\R^N}|\kappa(t) |^2dx \leq
(2+4C_b^2m_0^2)\Big(CI_{N+2}(0)+C_2(M,r) \int_0^t
\|\kappa(s)\|_2^2ds\Big).
\end{align}
Since the gauge transform $b$ does not depend of $\psi_0$ and for all $t\in [0,T]$,
$\|\psi_0(t)\|_{L^\infty(\R^N)}\leq m_0$, there is a constant
$C(m_0)$ depending only on $m_0$ such that
$$\sup_{t\in[0,T]}\|e^{-b(|\psi_0(t)|^2)}\|_{L^\infty(\R^N)}^2 \leq C(m_0).$$
Recalling that
$\kappa(0)=e^{p(|a_0|^2)}(1-\Delta)^{\frac{s}{2}}(e^{q(|a_0|^2)}
\Delta a_0)$ and choosing $m_{N+2}$ such that
\begin{eqnarray}\label{cauchy23}
m_{N+2}^2\geq 2C(m_0)(2+4C_b^2m_0^2)CI_{N+2}(0) +1,
\end{eqnarray}
one can find a positive $T$ such that for this choice of $m_{N+2}$
\begin{align}\label{cauchy24}
\sup_{t\in [0,T] }\|\phi_{N+2}\|_{H^{s}(\R^N)} \leq m_{N+2}.
\end{align}
Note that $m_{N+2}$ depends only on the initial data $a_0$ and
$m_0$. Performing the same kind of estimates on equations
\eqref{cauchy15}, one can find a positive $T$ and constant
$m_{N+1}$ depending only on $a_0$ and $m_0$ satisfying
\begin{align}\label{cauchy25}
m_{N+1}^2\geq 2C(m_0)(2+4C_b^2m_0^2)CI_{N+1}(0) +1,
\end{align}
where
\begin{align*}
 CI_{N+1}(0) &=\int_{\R^N}\frac{1+|\psi_0(0)|^2}{2+4|\psi_0(0)|^2}|\nu(0)|^2dx  \\
& +\int_{\R^N}
\Big(\frac{\psi_0^2(0)}{4(1+2|\psi_0(0)|^2)}\overline{\nu}^2(0)+\frac{\overline{\psi}_0^2(0)}{4(1+2|\psi_0(0)|^2)}\nu^2(0)\Big)dx
\end{align*}
with
$$\nu(0)=e^{p(|a_0|^2)}(1-\Delta)^{\frac{s}{2}}(e^{g(|a_0|^2)} \partial_t a_0),$$
such that
\begin{align}\label{cauchy26}
\sup_{t\in [0,T] }\|\phi_{N+1}\|_{H^{s}(\R^N)} \leq m_{N+1}.
\end{align}
Dealing with equation~\eqref{cauchy13}, we introduce for
$j=1,\dots,N$
$$\mu_j(0)=(1-\Delta)^{\frac{s}{2}}\partial_j a_0$$
and
\begin{align*}
CI_{j}(0) & =\int_{\R^N}\frac{1+|\psi_0(0)|^2}{2+4|\psi_0(0)|^2}|\mu_j(0)|^2dx   \\
&+\int_{\R^N}\Big(\frac{\psi_0^2(0)}{4(1+2|\psi_0(0)|^2)}\overline{\mu}_j^2(0)
+\frac{\overline{\psi}_0^2(0)}{4(1+2|\psi_0(0)|^2)}\mu_j^2(0)\Big)dx.
\end{align*}
Choosing $m_j$ depending only on $a_0$ and $m_0$ such that
\begin{align}\label{cauchy27}
m_{j}^2\geq 2(2+4C_b^2m_0^2)CI_{j}(0) +1,
\end{align}
we derive
\begin{align}\label{cauchy28}
 \text{ for } j=1,\dots,N, \sup_{t\in [0,T] }\|\phi_{j}\|_{H^{s}(\R^N)} \leq m_{j}.
\end{align}
Treating now equation \eqref{cauchy12}, we introduce
$$\xi(0)=(1-\Delta)^{\frac{s}{2}} a_0$$
and
$$ CI_{0}(0)=\int_{\R^N}|\xi(0)|^2dx.$$
It is crucial to remark here that equation~\eqref{cauchy12} is not
quasi-linear. As a consequence, we can perform a classical energy
estimate on it and choose the constant $m_0$ such that
\begin{align}\label{cauchy29}
m_{0}^2\geq 2CI_0(0) +1.
\end{align}
The choice of $m_0$ depends only on the initial data $a_0$.
\begin{remark}
 If we work with equation \eqref{cauchy3} instead of equation \eqref{cauchy8} and perform the energy estimates of equation \eqref{cauchy15} for example,
we have to choose $m_0$ such that
$$m_0^2\geq 2(2+4C_b^2m_0^2)CI_{0}(0) +1$$
where
\begin{align*}
\overset{\sim}{CI}_0(0) &=\int_{\R^N}\frac{1+|\psi_0(0)|^2}{2+4|\psi_0(0)|^2}|\xi(0)|^2dx   \\
& +\int_{\R^N}
\Big(\frac{\psi_0^2(0)}{4(1+2|\psi_0(0)|^2)}\overline{\xi}^2(0)+\frac{\overline{\psi}_0^2(0)}{4(1+2|\psi_0(0)|^2)}\xi^2(0)\Big)dx.
\end{align*}
Such a choice requires of course a smallness assumption on the
initial data $a_0$.
\end{remark}
Let us take $m_0$ as in \eqref{cauchy29}. Then one can find also a
positive $T$ such that
\begin{align}\label{cauchy30}
\sup_{t\in [0,T] }\|\phi_{0}\|_{H^{s}(\R^N)} \leq m_{0}.
\end{align}
We refer to \cite{co} for the technical details. Due to the
structure of the space $\mathcal{X}_T$, it remains to estimate
$(\psi_0)_t$ in $H^{E(\frac{N}{2})+1}(\R^N)$. This is done
directly on equation \eqref{cauchy12} and provides that there
exists a constant $C_0(M)$ depending only on $M$ such that
\begin{align}\label{cauchy31}
\sup_{t\in [0,T] }\|(\phi_0)_t\|_{H^{E(\frac{N}{2})+1}(\R^N)} \leq
C_0(M).
\end{align}
As a conclusion, we choose constants $M$, $r$ and $T$ as follows.
We first fix $m_0$ depending only on $a_0$ such that
\eqref{cauchy29} holds. Then we take $(m_j)$, $m_{N+1}$ and
$m_{N+2}$ depending only on $a_0$ and $m_0$ satisfying
respectively \eqref{cauchy27}, \eqref{cauchy25} and
\eqref{cauchy23}. Finally take $r$ such that
$$
r\geq C_0(M),
$$
and $T$ sufficiently small such that
$$
C_4(M,r)T\leq \frac{1}{2},
$$
and similar conditions to take into account the equations on
$\phi_0$, $\phi_j$ and $\phi_{N+1}$. For such a choice of
parameter, we have showed
$$\mathcal{S}\Big( \mathcal{X}_T(M,r)\Big) \subset  \mathcal{X}_T(M,r).$$
The fact that the mapping $\mathcal{S}$ is a contraction for the
suitable norm is very standard and we refer once again to
\cite{co} since the proof reads exactly the same. By the
contraction mapping principle, there exists a unique solution
$$
\Phi=\Big(\phi_0,(\phi_j)_{j=0}^N,\phi_{N+1},\phi_{N+2}\Big)
$$
to system~\eqref{cauchy12}-\eqref{cauchy15}. Furthermore, for each
$0\leq j\leq N+2$, the function $\phi_j$ satisfies
$$
\phi_j\in L^\infty(0,T;H^s(\R^N))\cap C([0,T];L^2(\R^N)).
$$
To conclude the proof, we have to show that the solution $\Phi$
solves system~\eqref{cauchy8}-\eqref{cauchy11} and has the
following regularity
$$
\Phi\in \Big(L^\infty(0,T;H^{s+2}(\R^N))\cap
C([0,T];H^{s}(\R^N))\Big)^{N+3}.
$$
This can be done exactly as in \cite{co}. The proof of the
conservation laws \eqref{massc}-\eqref{energyc} is very standard
once we have proved that $\phi$ is regular and so we omit it. At
this point the proof of Theorem~\ref{cauchy} is completed.

\medskip

\section{Existence of ground states and orbital instability}\label{instability}

In this section we derive the existence, as well as some qualitative properties, of the
ground states solutions of~\eqref{gs1}. When $p > 3 + \frac{4}{N}$
we shall also prove that the ground states are instable by blow-up.
\medskip

We begin with the following Pohozaev-type identity.

\begin{lemma}\label{poho}
Any $u\in X_{\C}$ solution of~\eqref{gs1} satisfies $P(u)=0$ where
$P: X_{\C} \to \R$ is the function defined by
\begin{equation}
P(u) = \frac{N-2}{N}\Big( \frac{1}{2}\int_{\R^N}|\nabla u|^2dx
+\int_{\R^N}|u|^2|\nabla |u||^2dx\Big)
+\frac{\omega}{2}\int_{\R^N}|u|^2dx-\frac{1}{p+1}\int_{\R^N}|u|^{p+1}dx.
\nonumber
\end{equation}
\end{lemma}
\begin{proof}
Since the proof only uses classical arguments, we shall just
sketch it and refer to~\cite{co2} for further details. Let $u \in
X_{\C}$ be a solution to equation~\eqref{gs1}. From~\cite[Section
6. Appendix]{LiWaWa2} we learn that $u\in L^\infty_{{\rm
loc}}(\R^N)$ (the proof given there extend easily to complex
valued functions). We are then able to pursue as
in~\cite[Proposition 2.1]{co2}. Let $\psi \in C_0^{\infty}(\R^N)$
be such that $\psi \geq 0$, $\text{supp}(\psi) \subset B(0,2)$ and
$\psi \equiv 1$ on $B(0,1)$. For all $j\in \N^*$, we set
$\psi_j(x)=\psi(\frac{x}{j})$.  Now let $(\rho_n)_{n\in \N}$ be a
sequence of even positive functions in $L^1(\R^N)$ with
$\int_{\R^N}\rho_n dx=1$ such that, for all $\kappa \in
L^q(\R^N)$, $\rho_n*\kappa$ tends to $\kappa$ in $L^q(\R^N)$, as
$n\to\infty$, for all $1\leq q <\infty$. First, we take the
convolution of~\eqref{gs1} with $\rho_n$. Then, we multiply the
resulting equation by $\psi_j \hspace{2mm}x\cdot \nabla
(\overline{
  u} *\rho_n)$, integrate over $\R^N$ and consider the
real part of the equality. From that point, the calculus
are standard consisting in various integrations by parts. Hence, we omit the
details and we refer the reader to~\cite{co2}. In order to conclude the
proof, it is sufficient to apply the Lebesgue dominated convergence
theorem.
\end{proof}

\vskip3pt \noindent {\it Proof of Theorem~\ref{ex}.} We shall distinguish
between the cases $N=1$ and $N \geq 2$,  which require a separate treatment.
\vskip2pt \noindent $\bullet$ Case $N \geq 2$. We divide the proof into four
steps. \vskip2pt \noindent {\bf Step I (existence of a solution
to~\eqref{gs1}).} We prove the existence of a ground state solution $u_{\omega}
\in X_{\C}$ to~\eqref{gs1} satisfying conditions {\rm i)-iv)} of
Theorem~\ref{ex}. Following the arguments of~\cite{cojean}, we perform a change
of unknown by setting $v=r^{-1}(u)$, where the function $r:\R\to\R$ is the
unique solution to the Cauchy problem
\begin{equation}
    \label{changevariable}
r'(s)=\frac{1}{\sqrt{1+2r^2(s)}},\qquad r(0)=0.
\end{equation}
Here $u \in X_{\C}$ is assumed to be real valued. Then, in~\cite{cojean} it
is proved that, if  $v \in H^1(\R^N) \cap C^2(\R^N)$ is a real
solution to
\begin{eqnarray}
    \label{exis1}
-\Delta v= \frac{1}{\sqrt{1+2r^2(v)}}\Big( |r(v)|^{p-1}r(v)-\omega
r(v)\Big),
\end{eqnarray}
then $u = r(v) \in X_{\C} \cap C^2(\R^N)$ and it is a real solution
of~\eqref{gs1}.  Let us set
$$
k(v):= \frac{1}{\sqrt{1+2r^2(v)}}\Big( |r(v)|^{p-1}r(v)-\omega r(v)\Big)=r'(v)\Big( |r(v)|^{p-1}r(v)-\omega r(v)\Big),
$$
and denote by ${\mathcal T}_{\omega} : H^1(\R^N) \to \R$ the
action associated with equation~\eqref{exis1}, namely
\begin{align*}
{\mathcal T}_{\omega}(v)&=\frac{1}{2} \int_{\R^N} |\nabla v|^2dx -\int_{\R^N} K(v)dx,\\
&=\frac{1}{2} \int_{\R^N} |\nabla v|^2dx -\frac{1}{p+1}\int_{\R^N}
|r(v)|^{p+1}dx+\frac{\omega}{2}\int_{\R^N} |r(v)|^2dx,
\end{align*}
where we have set $K(t)=\int_0^t k(s)ds$. Now, it is straightforward to check
that $k$ satisfies assumptions (g0)-(g3) of~\cite{cojean}. Thus,
from~\cite{cojean} (see also~\cite{bere1,bere2}) we deduce the existence of a
ground state $v_\omega$ of~\eqref{exis1} satisfying conditions {\rm i)-iv)} of
Theorem~\ref{ex}, that is $v_\omega$ solves~\eqref{exis1} and minimizes the
action ${\mathcal T}_{\omega}$ among all nontrivial solutions to~\eqref{exis1}.
Therefore, setting $u_{\omega} = r(v_\omega)$, we get that $u_\omega$
solves~\eqref{gs1} and satisfies conditions {\rm i)-iv)} of Theorem~\ref{ex}
(see~\cite[Theorem 1.2]{cojean}). \vskip4pt \noindent {\bf Step II (existence
of a ground state to~\eqref{gs1}).} In this step we prove that $u_\omega$
minimizes the action ${\mathcal E}_{\omega}$, over the set of nontrivial
solutions to the original equation~\eqref{gs1}. To achieve this goal, we make
the following observations. Notice first that, if $u=r(v)$ with $u \in X_{\C}$
real, then ${\mathcal E}_{\omega}(u)= {\mathcal T}_{\omega}(v)$. Indeed, we
have
\begin{align*}
{\mathcal E}_{\omega}(r(v))&=
\frac{1}{2}\int_{\R^N}r^{'2}(v)|\nabla
v|^2dx+\int_{\R^N}|r(v)|^2r^{'2}(v)|\nabla |v||^2dx
-\frac{1}{p+1}\int_{\R^N}|r(v)|^{p+1}dx\\ &+\frac{\omega}{2}\int_{\R^N} |r(v)|^2dx\\
&= \frac{1}{2}\int_{\R^N} \frac{1}{1+2r^2(v)}|\nabla
v|^2dx+\int_{\R^N}\frac{1}{1+2r^2(v)}r(v)^2|\nabla v|^2dx
-\frac{1}{p+1}\int_{\R^N}|r(v)|^{p+1}dx\\ &+\frac{\omega}{2}\int_{\R^N} |r(v)|^2dx\\
&= \frac{1}{2}\int_{\R^N}|\nabla v|^2dx
-\frac{1}{p+1}\int_{\R^N}|r(v)|^{p+1}dx+\frac{\omega}{2}\int_{\R^N}
|r(v)|^2dx\\ &= {\mathcal T}_{\omega}(v),
\end{align*}
thanks to the Cauchy problem~\eqref{changevariable}. Also, if $u \in X_{\C}$ is a
solution to~\eqref{gs1} we have, in light of Lemma~\ref{poho}, that
\begin{equation}\label{poho1}
{\mathcal E}_{\omega}(u) = \frac{1}{N}\int_{\R^N}|\nabla u|^2 + 2
|u|^2 |\nabla|u||^2 dx.
\end{equation}
Once these facts have been observed, take any solution $u \in
X_{\C}$ to~\eqref{gs1} (notice that $u$ can be a complex valued
function) and set $v=r^{-1}(|u|)$. Due to the well-known
point-wise inequality $|\nabla |u(x)||\leq |\nabla u(x)|$ for
a.e.\ $x\in\R^N$, it holds
\begin{equation}\label{semi}
\int_{\R^N} |\nabla |u(x)||^2dx\leq \int_{\R^N} |\nabla u(x)|^2dx,
\end{equation}
so that ${\mathcal E}_{\omega}(|u|) \leq {\mathcal E}_{\omega}(u)$
(notice that all the other terms in the functional ${\mathcal
E}_{\omega}$ are invariant to the modulus). Thus, in turn, we have
\begin{equation}
    \label{control}
{\mathcal E}_{\omega}(u) \geq {\mathcal E}_{\omega}(|u|) =
{\mathcal E}_{\omega}(r(v))= {\mathcal T}_{\omega}(v).
\end{equation}
Now, let us set
$$
A = \big\{v\in H^1(\R^N): \tilde{P}(v) =0 \big\},
$$
where $\tilde{P}: H^1(\R^N) \to \R$ is the functional defined as
$$
\tilde{P}(v) = (N-2) \int_{\R^N}|\nabla v|^2 dx - 2N\int_{\R^N}K(v) dx.
$$
Clearly, for any $v \in A$, we have
\begin{equation}\label{poho2}
{\mathcal T}_{\omega}(v)= \frac{1}{N}\int_{\R^N}|\nabla v|^2dx.
\end{equation}
Also, as for the proof that ${\mathcal E}_{\omega}(u)= {\mathcal
T}_{\omega}(v)$,  it is readily checked that, if $v= r^{-1}(u)$ with
$u \in X_{\C}$ real, then $\tilde{P}(v) = P(u)$. Finally, it is well
known (see e.g.~\cite{bere1,bere2}) that $v_\omega$ satisfies
\begin{equation}\label{exis2}
v_\omega \in A,\qquad
 {\mathcal T}_{\omega}(v_\omega)=\inf_{v\in
A}{\mathcal T}_{\omega}(v).
\end{equation}
Now, if $N=2$, it follows from the definition of $P$ in Lemma~\ref{poho}
that $P(|u|)=0$. Thus, in turn, $\tilde{P}(v) =0$ and, using~\eqref{control}
and~\eqref{exis2}, it follows that
\begin{equation}\label{good}
{\mathcal E}_{\omega}(u)\geq {\mathcal T}_{\omega}(v)\geq
{\mathcal T}_{\omega}(v_\omega)= {\mathcal E}_{\omega}(u_\omega),
\end{equation}
proving the desired claim.
If $N \geq 3$, one of the following possibilities occurs.
\begin{itemize}
\item[i)] $P(|u|) =0$. In this case inequality~\eqref{good} holds exactly as in the case $N=2$.
\item[ii)] $P(|u|) = \tilde{P}(v) < 0$. In this case there exists a number $\theta \in
(0,1)$ such that, setting $v_{\theta}(x)= v(x/\theta)$, we have
$\tilde{P}(v_{\theta})= 0$. Now, since $v_{\theta} \in A$,
using~\eqref{poho1},~\eqref{semi},~\eqref{poho2},~\eqref{exis2},
it follows that
\begin{align*}
{\mathcal T}_{\omega}(v_{\theta})&= \frac{1}{N}\int_{\R^N} |\nabla
v_{\theta}|^2 dx  = \frac{\theta^{N-2}}{N}\int_{\R^N}|\nabla v|^2
dx\\
&= \frac{\theta^{N-2}}{N} \int_{\R^N}|\nabla |u||^2 + 2 |u|^2
|\nabla |u||^2 dx \\
& \leq \frac{\theta^{N-2}}{N} \int_{\R^N}|\nabla u|^2 + 2 |u|^2
|\nabla |u||^2 dx \\
& = \frac{\theta^{N-2}}{N} N {\mathcal E}_{\omega}(u) =
\theta^{N-2}{\mathcal E}_{\omega}(u) < {\mathcal E}_{\omega}(u).
\end{align*}
Thus, we get
$$
{\mathcal E}_{\omega}(u) > {\mathcal T}_{\omega}(v_{\theta})
\geq {\mathcal T}_{\omega}(v_{\omega}) = {\mathcal
E}_{\omega}(u_{\omega}).
$$
\end{itemize}
Then, in conclusion, we proved that for both the cases $N=2$ and
$N\geq 3$, $u_{\omega} \in X_{\C}$ indeed minimizes the action
${\mathcal E}_\omega$ over the set of nontrivial solutions
to~\eqref{gs1}.
\vskip4pt
\noindent {\bf Step III (real character
of solutions).} First we prove that, if $u \in X_{\C}$ is a ground
state solution to~\eqref{gs1}, then $|u| \in X$ is also a ground
state. We set $v= r^{-1}(|u|)$. Observe that it holds
\begin{equation}
    \label{ok}
m_{\omega}= {\mathcal E}_{\omega}(u) \geq {\mathcal E}_{\omega}(|u|)\geq {\mathcal T}_{\omega}(v).
\end{equation}
In the case $N=2$, we have $\tilde{P}(v) = P(|u|) = 0$ and,
thus, we conclude ${\mathcal E}_{\omega}(|u|)=m_{\omega}$ by using~\eqref{exis2},~\eqref{ok} and recalling
that ${\mathcal T}_{\omega}(v_{\omega}) = {\mathcal
E}_{\omega}(u_{\omega})= m_{\omega}$. If $N \geq 3$,
 and $\tilde{P}(v) = P(|u|) <0$ we introduce, as before, the rescaling $v_{\theta}$
such that $\tilde{P}(v_{\theta}) = 0$. Then, we get
$$
{\mathcal T}_{\omega}(v_{\theta}) < {\mathcal E}_{\omega}(u)= m_{\omega},
$$
and we immediately reach a contradiction by arguing as before.
Now, let $u \in X_{\C}$ be a ground state solution of~\eqref{gs1}
and assume that
$$
{\mathcal L}^N(\{x\in\R^N:|\nabla |u|(x)|<|\nabla u(x)|\})>0.
$$
Then we get
\begin{align*}
m_{\omega} & = \frac{1}{2}\int_{\R^N} |\nabla |u||^2 dx +
\int_{\R^N}|u|^2 |\nabla |u||^2 dx +\frac{\omega}{2}\int_{\R^N}|u|^2dx- \frac{1}{p+1}\int_{\R^N} |u|^{p+1}dx \\
&< \frac{1}{2}\int_{\R^N} |\nabla u|^2 dx +  \int_{\R^N} |u|^2
|\nabla |u||^2 dx +\frac{\omega}{2}\int_{\R^N}|u|^2dx- \frac{1}{p+1}\int_{\R^N} |u|^{p+1}dx=
m_{\omega}.
\end{align*}
This is obviously not possible and, hence, we have $|\nabla
|u(x)||=|\nabla u(x)|$, for a.e.\ $x\in\R^N$. But this is true if,
and only if, $\Re\, u\nabla (\Im\, u)=\Im\, u\nabla(\Re\, u)$. Whence,
if this last condition holds, we get
$$
{\bar u}\nabla u=\Re\, u\nabla (\Re\, u)+ \Im\, u\nabla (\Im\,
u),\quad \text{a.e.\ in $\R^N$},
$$
which implies that $\Re\,(i\bar u(x)\nabla u(x))=0$ a.e.\ in
$\R^N$. This last identity immediately gives the existence of
$\theta\in {\mathbb S}^1$ such that $u(x)=e^{i\theta}|u(x)|$.
\vskip4pt \noindent {\bf Step IV (properties i)-iv) for any real
non negative ground state).} In light of some recent
achievements~\cite{Ma,BJM}, we can prove that any real ground
state solution to~\eqref{gs1} is radially symmetric and radially
decreasing about some point. In fact we observe first that for any
given solution $u$ of~\eqref{gs1}, by~\cite[Section 6.
Appendix]{LiWaWa2}, $u\in L^\infty_{{\rm loc}}(\R^N)$ and in turn
$u\in C^2(\R^N)$ (cf.~\cite{LU}). Considering now the strictly
increasing function $\mu:\R\to\R$ such that
\begin{equation}
    \label{changevariable-anti}
\mu'(s)=\sqrt{1+2s^2},\qquad \mu(0)=0,
\end{equation}
it is easy to see that $v=\mu(u)$ is a solution of~\eqref{exis1}.
Notice that $\mu$ is precisely the inverse function of the
function $r$ introduced in Step II, $r\circ \mu=\mu\circ r={\rm
Id}$. Furthermore, we claim that if $u$ is any given ground state
of~\eqref{gs1}, then $v=\mu(u)=r^{-1}(u)$ is a ground state
of~\eqref{exis1}. In fact, taking into account the computations in
Step II of the proof, for any nontrivial solution $w$
of~\eqref{exis1}, $r(w)$ is a (nontrivial) solution
of~\eqref{gs1}, and we have
$$
{\mathcal T}_\omega(w)={\mathcal E}_\omega(r(w))\geq
m_\omega={\mathcal E}_\omega(u) ={\mathcal
E}_\omega(r(v))={\mathcal T}_\omega(v),
$$
which yields the desired conclusion. At this point the fact that any ground
state solution is radially symmetric and radially decreasing about some point
is a consequence of the results of~\cite{BJM} (see also~\cite{MSJJ}) applied to
equation~\eqref{exis1}. Here let us point out that the radial symmetry (plus
radial decrease) could have also been proved by arguing directly on
equation~\eqref{gs1} which, in fact, satisfies a scaling property being the
essence of the results of~\cite{BJM}. Now let $u \in {\mathcal G_{\omega}}$ be
such that $u\geq 0$ in $\R^N$. Since $u \in C^2(\R^N)$  we have by the maximum
principle (applies to $v=\mu(u)$ ) that $u >0$ on $\R^N$. Finally
using~\cite[Lemma 2]{bere2} on equation~\eqref{exis1} we immediately derive the
exponential decays indicated in the statement of Theorem~\ref{ex}. \vskip5pt
\noindent $\bullet$ Case $N=1$. \vskip2pt \noindent By taking advantage of the
transformation of problem~\eqref{gs1}, via the dual approach, into the
semi-linear equation~\eqref{exis1}, we know that equation~\eqref{gs1} admits a
unique positive and even solution (see ~\cite[Theorem 5, Remark 6.3]{bere2}).
Thus it just remains to prove that any solution $u$ of~\eqref{gs1} is of the
form $u=e^{i \theta}\phi$, where $\theta \in \R$ and $\phi
>0$ is a solution to~\eqref{gs1}. In
fact $|u|>0$, otherwise we would get a contradiction with the
identity
$$
\frac{1}{2}|u'|^2+\frac{1}{4}|(|u|^2)'|^2-\frac{\omega}{2}|u|^2+\frac{1}{p+1}|u|^{p+1}=0.
$$
This identity is obtained multiplying~\eqref{gs1} by the conjugate
of $u'$ and by performing standard manipulations. Then, we can
write down the solution in polar form, $u=\rho e^{{\rm i}\theta}$,
where $\rho,\theta\in C^2(\R)$. By direct computation, it holds
$u''=\big[\rho\theta''+2\rho'\theta'\big]e^{{\rm i}\theta}{\rm i}+
\big[\rho''-\rho(\theta')^2\big]e^{{\rm i}\theta}$. Then, by
dropping this formula into equation~\eqref{gs1}, exactly as
in~\cite[proof of Theorem 8.1.7{\rm (iii)}]{Ca}, one immediately
reaches (by comparison of real and imaginary parts) the following
identity
\begin{equation}
    \label{polarequl}
\rho\theta''+2\rho'\theta'=0,
\end{equation}
namely $\theta'=\frac{K}{\rho^2}$, for some $K\geq 0$. At this point it is
sufficient to follow the argument of~\cite[proof of Theorem 8.1.7{\rm
(iii)}]{Ca} to prove that $K=0$ and get the desired property. Thus, when $N=1$,
Theorem~\ref{ex} holds true and the set of solutions of~\eqref{gs1} is
essentially unique. $\hfill\Box$

\bigskip
\noindent In the rest of this section we prove the instability
result, Theorem~\ref{instab}. We start with two preliminary
results. We define the variance $\V(t)$, by
\begin{equation}\label{eq.variance}
  \V(t)=\int_{\R^N} |x|^2|\phi(t,x)|^2\,dx,\quad t\in[0,\infty)
\end{equation}
and derive a so-called virial identity in the following lemma.

\begin{lemma}\label{lem.variance}
  Let $\phi$ be a solution of~\eqref{eq.schr1} on an interval $I=(-t_1,t_1)$. Then,
  \begin{equation}
    \V''(t)=8Q(\phi(t)),\quad t\in I,
    \label{eq.Vsecondo}
  \end{equation}
where we have set
\begin{equation}
    \label{defQ}
Q(\phi)=\int_{\R^N}|\nabla\phi|^2dx+(N+2)\int_{\R^N}
|\phi|^2|\nabla |\phi||^2dx
    -\frac{N(p-1)}{2(p+1)}\int_{\R^N} |\phi|^{p+1}dx,
\end{equation}
for all $\phi\in X_{\C}$.
\end{lemma}

\begin{proof}
We introduce the following notations:
\begin{equation*}
  z=  (z^1,\dots,z^n)\in\C^N;
  \quad
  z\cdot w=  \sum_{i=1}^N z^iw^i,\qquad z,w\in\C^N;
\end{equation*}
\begin{equation*}
  \phi_i=  \frac{\partial \phi}{\partial x_i},
  \qquad \phi:\R^N\to\C.
\end{equation*}
Let us first prove that
\begin{equation}
    \label{eq.Vprimo}
    \V'(t)=  4\Im\int_{\R^N}\bigl(x\cdot
    \nabla\phi\bigr)
    \overline{\phi}\,dx,\quad t\in I.
  \end{equation}
By multiplying equation~\eqref{eq.schr1} by $2\overline\phi$ and
taking the imaginary parts, yields
\begin{equation}
  \frac{\partial}{\partial t}|\phi|^2=
  -2\Im(\overline\phi\Delta\phi)=
  -2\nabla\cdot(\Im\overline\phi\nabla\phi),
  \label{eq.1}
\end{equation}
Now, multiplying \eqref{eq.1} by $|x|^2$, and integrating by parts
in space, we get~\eqref{eq.Vprimo}. In order to
prove~\eqref{eq.Vsecondo}, let us multiply
equation~\eqref{eq.schr1} by $2x\cdot\nabla\overline\phi$,
integrate in space on $\R^N$ and, finally, take the real parts,
yielding
\begin{align*}
  0 &=
  2\Re\int_{\R^N} \iu (x\cdot\nabla\overline\phi)\phi_t\,dx
  +2\Re\int_{\R^N} (x\cdot\nabla\overline\phi)\Delta\phi\,dx  \\
&+2\Re\int_{\R^N}
(x\cdot\nabla\overline\phi)\phi\Delta|\phi|^2\,dx
   +2\Re\int_{\R^N}(x\cdot\nabla\overline\phi)
  |\phi|^{p-1}\phi\,dx.
\end{align*}
We rewrite the last identity in the form
\begin{equation}\label{eq.eq}
  \hbox{I}=\hbox{II}+\hbox{III},
\end{equation}
where
\begin{align*}
  \hbox{I}= & 2\Re\int_{\R^N} \iu (x\cdot\nabla\overline\phi)\phi_t\,dx,
  \\
  \hbox{II}= & -2\Re\int_{\R^N} (x\cdot\nabla\overline\phi)\Delta\phi\,dx-2\Re\int_{\R^N} (x\cdot\nabla\overline\phi)\phi\Delta|\phi|^2\,dx,
  \\
  \hbox{III}= & -2\Re\int_{\R^N}(x\cdot\nabla\overline\phi)|\phi|^{p-1}\phi\,dx.
\end{align*}
For the first term, recalling formula~\eqref{eq.Vprimo} for $\V'$,
we have
\begin{align}
    \label{eq.Afin}
   \hbox{I} & =
  \Re\int_{\R^N} \iu\sum_{j=1}^N\left(x^j\overline\phi_j\phi_t
  -x^j\phi_j\overline\phi_t\right)\,dx=
  \Re\int_{\R^N} \iu\sum_{j=1}^N x^j\left[(\overline\phi_j\phi)_t-
  (\phi\overline\phi_t)_j\right]\,dx     \notag\\
   & =   \frac{d}{dt}\Re\int_{\R^N} \iu
  (x\cdot\nabla\overline\phi)
  \phi\,dx+N\Re\int_{\R^N} \iu \phi\overline\phi_t\,dx \\
  & =\frac{d}{dt}\Im\int_{\R^N}
  (x\cdot\nabla\phi)
  \overline\phi\,dx
  -N\int_{\R^N}|\nabla\phi|^2\,dx  \notag\\
&+N\int_{\R^N}|\phi|^2\Delta |\phi|^2\,dx
  +N\int_{\R^N}|\phi|^{p+1}\,dx.  \notag  \\
& =\frac{1}{4}\frac{d}{dt}\V(t)
  -N\int_{\R^N}|\nabla\phi|^2\,dx
-4N\int_{\R^N}|\phi|^2|\nabla |\phi||^2\,dx
  +N\int_{\R^N}|\phi|^{p+1}\,dx.  \notag
\end{align}
A multiple integration by parts in formula $\hbox{II}$ gives
\begin{equation}\label{eq.Bfin}
  \hbox{II}=
(2-N)\int_{\R^N}|\nabla\phi|^2dx+2(2-N)\int_{\R^N}|\phi|^2|\nabla
|\phi||^2 dx.
\end{equation}
As for the term $\hbox{III}$, we write it by components
\begin{align}\label{eq.C2}
  \hbox{III} & =  -\sum_{j=1}^N\int_{\R^N} x^j |\phi|^{p-1}
  (2\Re\overline\phi_j\phi)\,dx  \\
&=-2\sum_{j=1}^N\int_{\R^N} x^j\frac{\partial_j
|\phi|^{p+1}}{p+1}dx= \frac{2N}{p+1}\int_{\R^N}|\phi|^{p+1}\,dx.
\notag
\end{align}
Finally, recollecting \eqref{eq.eq}, \eqref{eq.Afin},
\eqref{eq.Bfin}, \eqref{eq.C2} and \eqref{eq.Vprimo}, and taking
into account the definition of $Q$, the proof of
\eqref{eq.Vsecondo} is complete.
\end{proof}

In our next preliminary result we establish some qualitative
properties of a class of $L^2$-invariant rescaling.

\begin{lemma}
    \label{rescprop}
    Let $\psi\in X_{\C}$ and $Q(\psi)\leq 0$ and assume that
    \begin{equation}
        \label{prange}
        p>3+\frac{4}{N}.
    \end{equation}
    Let $\sigma>0$ and define the rescaling $\psi^\sigma(x)=\sigma^{N/2}\psi(\sigma x)$. Then
    there exists $\sigma_0\in (0,1]$ such that following facts hold
    \begin{enumerate}
        \item $Q(\psi^{\sigma_0})=0$;
        \item $\sigma_0=1$ if and only if $Q(\psi)=0$;
        \item $\frac{\partial}{\partial\sigma}{\mathcal E}_{\omega}(\psi^\sigma)>0$ for $\sigma\in(0,\sigma_0)$, and
        $\frac{\partial }{\partial\sigma}{\mathcal E}_{\omega}(\psi^\sigma)<0$ for $\sigma\in(\sigma_0,\infty)$;
        \item $\sigma\mapsto {\mathcal E}_{\omega}(\psi^\sigma)$ is concave on $(\sigma_0,\infty)$;
        \item $\frac{\partial}{\partial\sigma}{\mathcal E}_{\omega}(\psi^\sigma)=\frac{Q(\psi^\sigma)}{\sigma}$.
    \end{enumerate}
\end{lemma}

\begin{proof}
By direct computation, we have
$$
{\mathcal
E}_{\omega}(\psi^\sigma)=\frac{\sigma^2}{2}\int_{\R^N}|\nabla
\psi|^2dx+ \sigma^{N+2}\int_{\R^N}|\psi|^2|\nabla |\psi||^2dx +
\frac{\omega}{2}\int_{\R^N}|u|^2 dx
-\frac{\sigma^{\frac{N(p-1)}{2}}}{p+1}\int_{\R^N}|\psi|^{p+1}dx,
$$
so that, using the functional $Q$ defined by~\eqref{defQ}, for all $\sigma>0$, we
get
\begin{align*}
\frac{\partial }{\partial\sigma}{\mathcal E}_{\omega}(\psi^\sigma)
&= \sigma\int_{\R^N}|\nabla \psi|^2dx+
(N+2)\sigma^{N+1}\int_{\R^N}|\psi|^2|\nabla |\psi||^2dx \\
&-\frac{N(p-1)}{2(p+1)}\sigma^{\frac{N(p-1)}{2}-1}\int_{\R^N}|\psi|^{p+1}dx
=\frac{1}{\sigma} Q(\psi^\sigma).
\end{align*}
Then, taking into account~\eqref{prange}, it is readily seen that
there exists $\sigma_0\in(0,1]$ such that
$$
Q(\psi^{\sigma_0})=\sigma_0\frac{\partial
}{\partial\sigma}{\mathcal
E}_{\omega}(\psi^\sigma)_{|_{\sigma=\sigma_0}}=0,
$$
as well as $\frac{\partial }{\partial\sigma}{\mathcal
E}_{\omega}(\psi^\sigma)>0$ for $\sigma\in(0,\sigma_0)$ and
$\frac{\partial}{\partial\sigma}{\mathcal
E}_{\omega}(\psi^\sigma)<0$ for $\sigma\in(\sigma_0,\infty)$.
Furthermore, writing $\sigma =t\sigma_0$ we have
\begin{align*}
\frac{\partial^2 }{\partial\sigma^2}{\mathcal
E}_{\omega}(\psi^\sigma) &= \int_{\R^N}|\nabla \psi|^2dx+
(N+2)(N+1)t^N \sigma_0^{N}\int_{\R^N}|\psi|^2|\nabla |\psi||^2dx \\
&-\frac{N(p-1)}{2(p+1)}\Big(\frac{N(p-1)}{2}-1\Big)t^{\frac{N(p-1)}{2}-2}\sigma_0^{\frac{N(p-1)}{2}-2}\int_{\R^N}|\psi|^{p+1}dx,\\
& = t^N \Big( \frac{1}{t^N}\int_{\R^N}|\nabla \psi|^2dx +
(N+2)(N+1)\sigma_0^{N}\int_{\R^N}|\psi|^2|\nabla |\psi||^2dx\\
& -
\frac{N(p-1)}{2(p+1)}\Big(\frac{N(p-1)}{2}-1\Big)t^{\frac{Np-3N-4}{2}}\sigma_0^{\frac{N(p-1)}{2}-2}
\int_{\R^N}|\psi|^{p+1}dx\Big).
\end{align*}
Since, of course, we have
\begin{align*}
\int_{\R^N}|\nabla \psi|^2dx +
(N+2)(N+1)\sigma_0^{N}\int_{\R^N}|\psi|^2|\nabla |\psi||^2dx\\
-
\frac{N(p-1)}{2(p+1)}\Big(\frac{N(p-1)}{2}-1\Big)\sigma_0^{\frac{N(p-1)}{2}-2}\int_{\R^N}|\psi|^{p+1}dx
\leq 0
\end{align*}
and $t >1$, it follows that the quantity inside the parenthesis is
negative. Hence the map $\sigma\mapsto {\mathcal
E}_{\omega}(\psi^\sigma)$ is concave on $(\sigma_0,\infty)$,
concluding the proof.
\end{proof}

In order to establish the instability of ground states we now show,
in the spirit of~\cite{lecoz} that they enjoy two additional
variational characterizations.  First, we have the following

\begin{lemma}\label{nehari}
Assume that $\omega >0$ and $3\leq p\leq \frac{3N+2}{N-2}$ if $N
\geq 3$ and $3 \leq p$ if $N=1,2$. Then the set of minimizers of
\begin{equation}
    \label{solproblemground}
d_{\omega}=\inf\{{\mathcal E}_{\omega}(u): {\mathcal
I}_{\omega}(u)=0\},
\end{equation}
where
$$
{\mathcal I}_{\omega}(u)=\int_{\R^N}|\nabla
\phi|^2dx+\omega\int_{\R^N}|\phi|^2dx +4\int_{\R^N}|\phi|^2|\nabla
|\phi||^2dx-\int_{\R^N}|\phi|^{p+1}dx.
$$
is exactly the set of ground state ${\mathcal G_{\omega}}$. In
addition the value of the two infimums are equal.
\end{lemma}
\begin{proof}
First we show that if $u \in X_{\C}$ is a minimizer of
$d_{\omega}$ then $|u| \in X$ is also a minimizer of $d_{\omega}$.
Let $u \in X_{\C}$ with
    ${\mathcal
I}_{\omega}(u)=0$. Then ${\mathcal E}_{\omega}(|u|) \leq {\mathcal
E}_{\omega}(u)$ as well as
    ${\mathcal I}_{\omega}(|u|) \leq {\mathcal I}_{\omega}(u) = 0$. In particular and since $P \geq 3$, there
    exists $t \in (0,1]$ such that ${\mathcal I}_{\omega}(t|u|) = 0$.  Observe now that,
    for all $v\in X_{\C}$ such that ${\mathcal I}_{\omega}(v)=0$, it holds
    $$
    {\mathcal E}_{\omega}(v) = \frac{p-1}{2(p+1)}\int_{\R^N}|\nabla
    v|^2 dx + \frac{p-3}{p+1}\int_{\R^N}|v|^2 |\nabla |v||^2 dx +
    \omega \frac{p-1}{2(p+1)}\int_{\R^N}|v|^2 dx.
    $$
    Thus, since $p\geq 3$, it is readily seen that
    $$
    0 < {\mathcal E}_{\omega}(t|u|)\leq t^2 {\mathcal E}_{\omega}(u).
    $$
    In particular, if $u \in X_{\C}$ is a complex minimizer of $d_{\omega}$, then we have
    $$
{\mathcal E}_{\omega}(u)=d_{\omega}=\inf_{{\mathcal I}_{\omega}=0}
{\mathcal E}_{\omega}(\phi)\leq {\mathcal E}_{\omega}(t|u|)\leq
t^2 {\mathcal E}_{\omega}(u).
    $$
    Now, recalling that ${\mathcal E}_{\omega}(u)>0$ and $t\leq 1$, we immediately get $t=1$. Thus
    ${\mathcal I}_{\omega}(|u|)= {\mathcal I}_{\omega}(u)$ and in turn
    ${\mathcal E}_{\omega}(|u|)={\mathcal E}_{\omega}(u)$ proving that $|u| \in X$ is also a minimizer. Obviously
     it is only possible if the set $\{x\in\R^N:|\nabla |u|(x)|\neq |\nabla u(x)|\}$ has zero Lebesgue
    measure, which in turn implies that $u=|u|e^{{\rm i}\theta}$, for some $\theta\in{\mathbb S}^1$
    (see e.g. Step III of the proof of Theorem~\ref{ex}).
Now, when ${\mathcal E}_{\omega}$ is considered over $X$,
in~\cite[Theorem 1.1]{LiWaWa2} it is established that there exists
a nontrivial solution to the minimization
problem~\eqref{solproblemground} and that this minimizer is a
solution to equation~\eqref{gs1} (cf.~\cite[Lemma 2.5]{LiWaWa2}).
Clearly, since any minimizer is of the form  $u=|u|e^{{\rm
i}\theta}$ it is also solution to equation~\eqref{gs1}. Now, any
element $u \in X$ of ${\mathcal G_{\omega}}$ must satisfy
${\mathcal I}_{\omega}(u)=0$ and thus we deduce that the set of
ground states ${\mathcal G_{\omega}}$ and the set of minimizer
of~\eqref{solproblemground} coincide and that the values of the
two infimum values are equal.
\end{proof}

We also have the following

\begin{lemma}
Let us set
$$
c_{\omega}=\inf\{{\mathcal E}_{\omega}(\phi):\phi\in{\mathcal M}\}
\mbox{ where } {\mathcal M}=\{\phi\in X\setminus\{0\}:\,\,
Q(\phi)=0,\,\, {\mathcal I}_{\omega}(\phi)\leq 0\}.
$$
Then $c_{\omega}=d_{\omega}\, ( = m_{\omega}).$
\end{lemma}

\begin{proof}
Let $u \in X_{\C}$ be a solution to~\eqref{solproblemground}. By
Lemma~\ref{nehari} it is a ground state solution of~\eqref{gs1}
and applying the virial identity~\eqref{eq.Vsecondo} to a standing
wave solution we immediately deduce that $Q(u)=0$. By definition
${\mathcal I}_{\omega}(u)=0$ and thus we have $u\in {\mathcal M}$.
Hence $c_{\omega}\leq d_{\omega}$, since ${\mathcal
E}_{\omega}(u)=d_{\omega}$. On the other hand, given
$\phi\in{\mathcal M}$, either ${\mathcal I}_{\omega}(\phi)=0$ (so
that ${\mathcal E}_{\omega}(\phi)\geq d_{\omega}$) or ${\mathcal
I}_{\omega}(\phi)<0$. In this second case, if $\sigma>0$ and we
consider the rescaling $\phi^\sigma(x)=\sigma^{N/2}\phi(\sigma
x)$, we have ${\mathcal I}_{\omega}(\phi^1)<0$ and
\begin{align*}
\lim_{\sigma\to 0^+}{\mathcal I}_{\omega}(\phi^\sigma)
&=\lim_{\sigma\to 0^+}\Big(
\sigma^2\int_{\R^N}|\nabla \phi|^2dx+\omega\int_{\R^N}|\phi|^2dx   \\
& +4\sigma^{N+2}\int_{\R^N}|\phi|^2|\nabla
|\phi||^2dx-\sigma^{\frac{N(p-1)}{2}}\int_{\R^N}|\phi|^{p+1}dx\Big)>0.
\end{align*}
In turn, one can find $\hat\sigma\in(0,1)$ such that ${\mathcal
I}_{\omega}(\phi^{\hat\sigma})=0$. Then, we get ${\mathcal
E}_{\omega}(\phi^{\hat\sigma})\geq d_{\omega}$. Since $Q(\phi)=0$
and $\|\phi\|_2=\|\phi^{\hat\sigma}\|_2$, from
Lemma~\ref{rescprop} we obtain $ {\mathcal E}_{\omega}(\phi)\geq
{\mathcal E}_{\omega}(\phi^{\hat\sigma})\geq d_{\omega}$. Whence $
{\mathcal E}_{\omega}(\phi)\geq d_{\omega}$ holds true for any
$\phi\in{\mathcal M}$, which yields $c_{\omega}\geq d_{\omega}$,
proving the claim.
\end{proof}

\vskip3pt \noindent {\it Proof of Theorem~\ref{instab}.}
 Let $\varepsilon>0$ be fixed and
consider $u^\sigma(x)=\sigma^{N/2}u(\sigma x)$ for the ground
state solution $u$. We have $\|u\|_2=\|u^{\sigma}\|_2$ and
by the continuity of the mapping $\sigma \mapsto \sigma^{N/2}u(\sigma x)$,
it is clear that, for $\sigma$ sufficiently close to 1, $\|u-u^\sigma \|_{H^1(\R^N)}\leq \varepsilon$.
Furthermore,
\begin{equation}
        \label{invset0}
{\mathcal E}_{\omega}(u^\sigma)<{\mathcal E}_{\omega}(u),\quad
Q(u^\sigma)<0,\quad {\mathcal I}_{\omega}(u^\sigma)<0,
\end{equation}
provided that $\sigma>1$ is sufficiently close to $1$. The first
two inequalities just follow by Lemma~\eqref{rescprop}. Concerning
the last one, it holds
\begin{align*}
    & {\mathcal
I}_{\omega}(u^\sigma) =2{\mathcal
E}_{\omega}(u^\sigma)+\frac{2}{N}Q(u^\sigma)-
    \frac{4}{N}\int_{\R^N}|u^\sigma|^2|\nabla u^\sigma|^2dx-\frac{2}{N}\int_{\R^N}|\nabla u^\sigma|^2dx \\
    &\leq 2{\mathcal E}_{\omega}(u)+\frac{2}{N}Q(u)-{\mathcal
I}_{\omega}(u)
    -\frac{4\sigma^{N+2}}{N}\int_{\R^N}|u|^2|\nabla u|^2dx-\frac{2\sigma^2}{N}\int_{\R^N}|\nabla u|^2dx \\
    &=\frac{4}{N}\int_{\R^N}|u|^2|\nabla u|^2dx-\frac{2}{N}\int_{\R^N}|\nabla u|^2dx
    -\frac{4\sigma^{N+2}}{N}\int_{\R^N}|u|^2|\nabla u|^2dx-\frac{2\sigma^2}{N}\int_{\R^N}|\nabla u|^2dx \\
    & =\frac{4}{N}(1-\sigma^{N+2})\int_{\R^N}|u|^2|\nabla u|^2dx
    +\frac{2}{N}(1-\sigma^2)\int_{\R^N}|\nabla u|^2dx<0.
\end{align*}
Now fixing a $\sigma >1$ such that~\eqref{invset0} hold we
approximate $u^{\sigma} \in X_{\C}$ by a function $v \in
C_0^{\infty}(\R^N) \subset H^{s+2}(\R^N)$ in such a way that we
still have $\|v\|_2=\|u\|_2$, $\|v-u\|_{H^1(\R^N)}\leq \varepsilon$ and
\begin{equation}
        \label{invset00}
{\mathcal E}_{\omega}(v)<{\mathcal E}_{\omega}(u),\quad
Q(v)<0,\quad {\mathcal I}_{\omega}(v)<0.
\end{equation}
This comes from the fact that by Theorem \ref{ex}, the ground state $u$ and
then $u^\sigma$ are bounded as well as their derivatives up to order 2. Then  direct
estimates on $\mathcal{E}_\omega$, $Q$ and $\mathcal{I}_\omega$ give the desired
inequalities \eqref{invset00}.
Assume that $\phi(t)$ is the solution of~\eqref{eq.schr1} with
initial data $\phi(0)=v$.
    Then, we claim
    \begin{equation}
        \label{invset}
{\mathcal E}_{\omega}(\phi(t))<{\mathcal E}_{\omega}(u),\quad
Q(\phi(t))<0,\quad {\mathcal I}_{\omega}(\phi(t))<0,\quad\text{for
all $t\in [0,T_{{\rm max}})$},
\end{equation}
$T_{{\rm max}}\in(0,\infty]$ being the maximal existence time.
First, due to the conservation of the energy and~\eqref{invset0},
we get
$$
{\mathcal E}_{\omega}(\phi(t))= {\mathcal
E}_{\omega}(v)<{\mathcal E}_{\omega}(u), \quad\text{for all
$t\in [0,T_{{\rm max}})$}.
$$
In turn, it follows immediately that ${\mathcal
I}_{\omega}(\phi(t))\neq 0$ for all $t\in [0,T_{{\rm max}})$.
Hence ${\mathcal I}_{\omega}(\phi(t))<0$ for all $t\in [0,T_{{\rm
max}})$ since it is negative for $t=0$. Similarly, $Q(\phi(t))\neq
0$ for all $t\in [0,T_{{\rm max}})$, otherwise if $Q(\phi(t_0))=0$
for some $t_0\in [0,T_{{\rm max}})$, we would have $\phi(t_0)\in
{\mathcal M}$, yielding ${\mathcal E}_{\omega}(\phi(t_0))\geq
{\mathcal E}_{\omega}(u)$ which contradicts the first inequality
of~\eqref{invset}. Hence $Q(\phi(t))<0$ for all $t\in [0,T_{{\rm
max}})$ as it is negative for $t=0$, concluding the proof
of~\eqref{invset}. \vskip4pt \noindent Let now $\psi=\phi(t)$ be
the solution to~\eqref{eq.schr1} at a fixed time $t\in (0,T_{{\rm
max}})$ and let $\psi^\sigma$ be the usual $L^2$-invariant
rescaling. We know that $Q(\psi)<0$. Hence there exists
$\tilde\sigma\in(0,1)$ such that $Q(\psi^{\tilde\sigma})=0$. If
${\mathcal I}_{\omega}(\psi^{\tilde\sigma})\leq 0$ we do not
change the value of $\tilde\sigma$, otherwise we pick
$\hat\sigma\in (\tilde\sigma,1)$ such that ${\mathcal
I}_{\omega}(\psi^{\hat\sigma})=0$. In any case, one obtains $
{\mathcal E}_{\omega}(\psi^{\tilde\sigma})\geq d_{\omega}$ and
$Q(\psi^{\tilde\sigma})\leq 0$. Therefore, by Lemma \ref{rescprop}
\begin{align*}
 {\mathcal E}_{\omega}(v) & = {\mathcal E}_{\omega}(\psi)\geq
{\mathcal E}_{\omega}(\psi^{\tilde\sigma})+ (1-\tilde\sigma)
\frac{\partial}{\partial\sigma}{\mathcal E}_{\omega}(\psi^{\sigma})_{|_{\sigma=1}}   \\
& = {\mathcal
E}_{\omega}(\psi^{\tilde\sigma})+(1-\tilde\sigma)Q(\psi)>d_{\omega}+Q(\psi).
\end{align*}
Putting $\varrho_0:=d_\omega-{\mathcal E}_{\omega}(v)>0$,
concluding we have
$$
Q(\phi(t))\leq -\varrho_0,\quad\text{for all $t\in [0,T_{{\rm
max}})$}.
$$
Finally, assuming that $T_{\text{max}}=+\infty$ and
using the virial identity of Lemma~\ref{lem.variance},
we obtain
$$
0<\V(t)\leq \V(0)+\V'(0)t-4\varrho_0 t^2
$$
which yields a contradiction taking $t$ sufficiently large. Then
$0<T_{\text{max}}<+\infty$ and the solution blows-up in finite time. This
concludes the proof. \qed

\medskip

\section{Stationary solutions with prescribed $L^2$ norm}\label{stationary-problems2}

In this section we study the minimization problem~\eqref{defvalc}.
We prove the existence of a minimizer when $ 1 < p < 3 +
\frac{4}{N}$ and $ m(c) <0$. We also discuss the condition $m(c)
<0$ and we prove Theorem~\ref{no-minimizer}. Consider the
(complex) minimization problem
\begin{eqnarray}
    \label{0.1-bis}
\mbox{ minimize} \quad {\mathcal E} \quad \mbox{ on } \quad
\|u\|_2^2= c,
\end{eqnarray}
where $c$ is a positive number. We have the following result.

\begin{proposition}\label{complex}
Let $v$ be a solution to the minimization problem~\eqref{0.1-bis}.
Then
\begin{equation*}
v(x)=e^{i\theta}|v(|x|)|,\quad x\in\R^N,
\end{equation*}
for some $\theta\in {\mathbb S}^{1}$. In particular, the solutions
of problem~\eqref{0.1-bis} are, up to a constant complex phase,
real-valued positive and radially symmetric.
\end{proposition}

\begin{proof}
    The proof has some similarities with the final part of the proof of
    Theorem~\ref{ex} so we will be brief here. Let $X$ denote again the
    restriction of $X_\C$ to real-valued functions. We set
\begin{equation*}
\sigma_\C=\inf\big\{{\mathcal E}(v): v\in X_\C,
\,\|v\|_2^2=c \big\},\quad
\sigma_\R=\inf\big\{{\mathcal E}(v): v\in X, \|v\|_2^2=c \big\}.
\end{equation*}
Let us prove that $\sigma_\C=\sigma_\R$. Trivially one has
$\sigma_\C\leq \sigma_\R$, since $X\subset X_\C$. Moreover, if
$v\in X_\C$, we see using~\eqref{semi} that ${\mathcal E}(|v|)\leq
{\mathcal E}(v)$. In particular, we conclude that
$\sigma_\R\leq\sigma_\C$, yielding the desired equality
$\sigma_\C=\sigma_\R$. Now let $v\in X_\C$ be a solution to
$\sigma_\C$ and assume by contradiction that the Lebesgue measure
${\mathcal L}^N$ of the set $\{x\in\R^N:|\nabla |v|(x)|<|\nabla
v(x)|\}$ is positive. Then, of course,
$\||v|\|_{2}^2=\|v\|_{2}^2=c$, and
\begin{align*}
\sigma_\R & \leq \frac{1}{2}\int_{\R^N} |\nabla |v||^2 dx +
\int_{\R^N}
|v|^2 |\nabla |v||^2 dx - \frac{1}{p+1}\int_{\R^N} |v|^{p+1}dx \\
&< \frac{1}{2}\int_{\R^N} |\nabla v|^2 dx +  \int_{\R^N} |v|^2
|\nabla |v||^2 dx - \frac{1}{p+1}\int_{\R^N}
|v|^{p+1}dx=\sigma_\C,
\end{align*}
contradicting equality $\sigma_\C=\sigma_\R$. Hence, we have
$|\nabla |v(x)||=|\nabla v(x)|$ for a.e.\ $x\in\R^N$ and as in the
proof of Theorem~\ref{ex} this gives the existence of $\theta\in
{\mathbb S}^1$ such that $v=e^{i\theta}|v|$. Finally the result of
radial symmetry is a direct consequence of~\cite[Theorem 2]{Ma}.
\end{proof}

From Proposition~\ref{complex} we deduce that it is sufficient to study the
(real) minimization problem
\begin{equation}
    \label{0.1}
\mbox{ minimize} \quad {\mathcal E} \quad \mbox{ on } \quad
\|u\|_{2}^2= c \quad \mbox{with } u \in X.
\end{equation}
for a positive number $c$. We set
\begin{equation}
    \label{defvalcbis}
m(c)=\inf \{{\mathcal E}(u):\, \text{$u\in X$,\, $\|u\|_{2}^2 =
c$}\}.
\end{equation}

\begin{lemma}
    \label{boundedness} We have
    \begin{enumerate}
        \item Assume that $1< p < 3 + \frac{4}{N}$. Then $m(c) > - \infty$ for any
        $c>0$. In addition if $(u_n) \subset X$ is any
minimizing sequence for problem~\eqref{0.1} then $(u_n)$ is
bounded in $X$ and the sequence
\begin{equation}
    \label{bddquasilsequence}
\int_{\R^N}|u_n|^2|\nabla u_n|^2 dx =
\frac{1}{4}\int_{\R^N}|\nabla (u_n^2)|^2 dx
\end{equation}
is bounded in $\R$.
            \vskip4pt
            \item In the case $p= 3+\frac{4}{N}$ the
same conclusions hold provided that $c>0$ is sufficiently small.
        \vskip4pt
        \item Assume that $3+ \frac{4}{N} < p <  \frac{4N}{N-2}$. Then $m(c)=-\infty $ for any $c>0$.
        \end{enumerate}
\end{lemma}

\begin{proof}
Notice that, using H\"{o}lder and Sobolev inequalities we have for
$$
\theta = \frac{(p-1)(N-2)}{2(N+2)}
$$
and some $K>0$ depending only on $N$, that for any $u \in X$
\begin{align}
    \label{1.1}
    \int_{\R^N}|u|^{p+1} dx  &\leq   \Big( \int_{\R^N} |u|^2 dx\Big)^{1 -
    \theta}
    \Big( \int_{\R^N}|u|^{\frac{4N}{N-2}}dx\Big)^{\theta} \nonumber \\
    & \leq  K \Big( \int_{\R^N} |u|^2 dx\Big)^{1 - \theta}
    \Big( \int_{\R^N}|u|^2 |\nabla u|^2 dx \Big)^{\frac{\theta
    N}{N-2}}.
\end{align}
Here we have used the fact that
$$
\int_{\R^N} |u|^{\frac{4N}{N-2}}dx = \int_{\R^N}
|u^2|^{\frac{2N}{N-2}}dx,\qquad \int_{\R^N} |\nabla(u^2)|^2dx = 4
\int_{\R^N} |u|^2 |\nabla u|^2 dx.
$$
From~\eqref{1.1} we get that
$$
{\mathcal E}(u) \geq \int_{\R^N}|u|^2|\nabla u|^2 dx -
\frac{1}{p+1}K c^{1- \theta} \Big( \int_{\R^N}|u|^2 |\nabla u|^2
dx \Big)^{\frac{\theta N}{N-2}}.
$$
If we assume that $p < 3 +  \frac{4}{N}$, we see that
$\frac{\theta N}{N-2} <1$ and thus the
sequence~\eqref{bddquasilsequence} is bounded in $\R$.
From~\eqref{1.1} we then got that $(\|u_n\|_{p+1})$ is bounded and
thus also that $(\|\nabla u_n\|_2)$ is bounded. This proves Point
(1). In the limit case $p= 3 +  \frac{4}{N}$ we still reach the
boundedness result for any positive $c$ such that $Kc^{1-
\theta}<p+1$, where $K,\theta>0$ are the numbers introduced in the
proof. Now for point (3) we fix $c >0$ and take
    $u \in X$ such that $\|u\|_2^2 = c$. Then, considering the scaling,
$$
\sigma \mapsto u^\sigma(x)=\sigma^{\frac{N}{2}}u(\sigma x),
$$
we get, for all $\sigma>0$,
$$
\int_{\R^N}|u^\sigma|^2 dx =  \int_{\R^N}|u|^2 dx = c,\qquad
\int_{\R^N}|\nabla u^\sigma|^2 dx =  \sigma^2 \int_{\R^N}|\nabla
u|^2 dx,
$$
$$
\int_{\R^N}|u^\sigma|^{p+1} dx =
\sigma^{\frac{N(p-1)}{2}}\int_{\R^N}|u|^{p+1} dx,\qquad
\int_{\R^N}|u^\sigma|^2 |\nabla u^\sigma|^2 dx =  \sigma^{(N+2)}
\int_{\R^N}|u|^2 |\nabla u|^2 dx.
$$
Thus $\|u^\sigma\|_2^2 = c$ for all $\sigma>0$ and
$$
{\mathcal E}(u^\sigma) = \frac{\sigma^2}{2}\int_{\R^N}|\nabla u|^2
dx + \sigma^{(N+2)} \int_{\R^N}|u|^2 |\nabla u|^2 dx -
\frac{\sigma^{\frac{N(p-1)}{2}}}{p+1}\int_{\R^N}|u|^{p+1} dx.
$$
Now just notice that, in the range $3+ \frac{4}{N} < p <
\frac{4N}{N-2}$ the dominant term is
$$
\frac{\sigma^{\frac{N(p-1)}{2}}}{p+1}\int_{\R^N}|u|^{p+1} dx.
$$ Thus ${\mathcal E}(u^\sigma) \to - \infty$ as $\sigma\to+\infty$.
This concludes the proof of (3).
\end{proof}

Concerning the existence of a minimizer we first show

\begin{lemma}
    \label{weak-convergence}
Assume that $1 < p < 3 + \frac{4N}{N-2}$. The following facts
hold.
\begin{enumerate}
\item If $u_n \rightharpoonup u$ in $X$ then setting
$$
T(u) = \frac{1}{2}\int_{\R^N}|\nabla u|^2 dx + \int_{\R^N}|u|^2
|\nabla u|^2 dx,
$$ we have
$$
T(u) \leq \liminf_{n \to \infty}T(u_n).
$$
\item For any $u \in X$  there exists a
Schwarz symmetric function $u^* \in X$ satisfying
$$
T(u^*) \leq T(u), \quad \int_{\R^N}|u^*|^2 dx =
\int_{\R^N}|u|^2dx, \quad \int_{\R^N}|u^*|^{p+1} dx =
\int_{\R^N}|u|^{p+1} dx.
$$
\item Let $(u_n) \subset X$ be a minimizing sequence for~\eqref{0.1}
of Schwartz symmetric functions  satisfying $u_n \rightharpoonup u
$ in $X$. Then we have
\begin{equation}
    \label{lowsemicconcl}
{\mathcal E}(u) \leq \liminf_{n \to \infty }{\mathcal E}(u_n) =
m(c).
\end{equation}
\end{enumerate}
\end{lemma}

\begin{proof} Point (1) is standard. Defining $j:[0,\infty) \times [0, \infty) \to \R$ by $j(s, \xi) =
\frac{1}{2}\xi^2 + s^2 \xi^2,$ then $\{ \xi \to j(s, \xi)\}$ is
convex and thus the result follows from classical results of A.
Ioffe (see e.g.\ \cite{Io1,Io2}). Concerning assertion (2) all we
need to prove is $T(u^*) \leq T(u)$, which follows
from~\cite[Corollary 3.3]{HaSq}. For Point (3), we claim that
\begin{equation}\label{strong-convergence}
 \int_{\R^N} |u_n|^{p+1} dx \to \int_{\R^N}|u|^{p+1} dx
\end{equation}
as $n\to\infty$. In fact, since $(u_n) \subset X$ is minimizing we
have, by Lemma~\ref{boundedness}, Point (1) that $\nabla (u_n^2)$
is uniformly bounded in $L^2(\R^N)$ and thus by the Sobolev
embedding $\sup_{n\in\N}\|u_n^2\|_{\frac{2N}{N-2}}<\infty$, which
gives $\sup_{n\in\N}\|u_n\|_{\frac{4N}{N-2}}<\infty$. Now, using
the fact that $(u_n) \subset X$ consists of radial decreasing
functions, from the radial Lemma A.IV of~\cite{bere2}, we deduce
that $(u_n)$ has a uniform decay at infinity (with respect to both
$n\in \N$ and $|x|$) and this shows, by standard argument,
that~\eqref{strong-convergence} hold. Now we conclude observing
that, from point (1),  $T(u) \leq \liminf_{n \to \infty} T(u_n)$.
\end{proof}

We now prove the existence of a minimizer for problem~\eqref{0.1}.

\begin{lemma}\label{general-minimizer}
Assume that $1< p < 3 + \frac{4}{N}$ and let $c >0$ be such that
$m(c)<0$. Then the problem~\eqref{0.1} admits a minimizer which is
Schwartz symmetric.
\end{lemma}

\begin{proof}
    Let $(u_n)$ be a minimizing sequence for (\ref{0.1}).
By Lemma~\ref{weak-convergence} we know that $(u_n) \subset X$ can
be replaced by a minimizing sequence $(u_n^*) \subset X$ of
Schwarz symmetric functions such that $u_n^* \rightharpoonup u^*$
and
\begin{equation}
    \label{weak-limit}
{\mathcal E}(u^*) \leq \liminf_{n \to \infty}{\mathcal E}(u_n^*) =
m(c).
\end{equation}
We still denote $u^*$ by $u$. To conclude we just need to prove
that $\|u\|_2^2 = c$. Since ${\mathcal E}(u) \leq m(c) <0$
necessarily  $u\neq 0$. Assume thus that $0 < \|u\|_2^2 = \lambda
< c$ and consider the scaling $v(x) = u(\sigma^{- \frac{1}{N}}x)$
for $\sigma>1$. Then $\|v\|_2^2 = \sigma \lambda$ and for $\sigma
= \frac{c}{\lambda}$ we have $\|v\|_2^2 = c$. Now we also get that
$$
{\mathcal E}(v) = \sigma^{1- \frac{2}{N}} \left[
\int_{\R^N}\frac{1}{2}|\nabla u|^2 + |\nabla u|^2 |u|^2 dx \right]
- \frac{\sigma}{p+1} \int_{\R^N}|u|^{p+1}dx.
$$
Thus, since $\sigma>1$ and ${\mathcal E}(u) <0$ we conclude that
${\mathcal E}(v) < {\mathcal E}(u)$, which is a contradiction.
This proves that $\|u\|_2^2 = c$ and thus~\eqref{0.1} admits a
minimizer. Finally, observe that, since $\|u_n^*\|_{p+1}
\to\|u^*\|_{p+1}$ as $n\to\infty$, necessarily $\|\nabla u_n^*\|_2
\to \|\nabla u^*\|_2$ as $n\to\infty$ and we deduce that the
Schwarz symmetric sequence strongly converges to $u^* \in X$.
\end{proof}
\medskip

We now start to discuss the condition $m(c) <0$.

\begin{lemma}\label{negative}
We have
\begin{enumerate}
        \item Assume that $1< p < 1 + \frac{4}{N}$. Then $m(c) <0$ for any
        $c>0$.
        \vskip4pt
        \item Assume that $1 + \frac{4}{N} \leq  p < 3 + \frac{4}{N}$. Then $
         m(c) \leq 0$ for any $c>0$. This inequality also hold
        if $p= 3 + \frac{4}{N}$ and $c>0$ is small.
        \vskip4pt
        \item Assume that $1 + \frac{4}{N} \leq p <  3 +\frac{4}{N}$. Then
        there exists a $c>0$, sufficiently large, such that $m(c) <0$.
        \end{enumerate}
\end{lemma}
\begin{proof}
For Points (1) and (2) we use the scaling introduced in the proof of
Lemma~\ref{boundedness}, Point (3). When $p < 1 + \frac{4}{N}$ we see that the
dominant term, as $\sigma \to 0^+$, is
$$
\frac{\sigma^{\frac{N(p-1)}{2}}}{p+1}\int_{\R^N}|u|^{p+1} dx,
$$
and this proves Point (1).  For Point (2), since ${\mathcal E}(u^\sigma)\to 0$
as $\sigma\to 0^+$, we have directly have that $ m(c) \leq 0$ for any $c>0$.
Now for Point (3) we consider, for a fixed $R
>0$, the radial function $w_R \in H^1(\R^N)$ defined by
$$
w_R(r):=
\begin{cases}
    1     & \text{if $r \leq R$}, \\
    1+ R -r   & \text{if $R \leq r \leq R+1 $},  \\
    0    & \text{if $r \geq R+1.$}
\end{cases}
$$
Integrating in radial coordinates, we have
$$
\int_{\R^N} |w_R(|x|)|^2 dx = C_N R^N + \varepsilon_1(R^{N-1}),
$$
where $\varepsilon_1(R^{N-1})/R^N \to 0 $, as $ R \to \infty.$
Also
$$
\int_{\R^N}|w_R(|x|)|^{p+1} dx =  C_N R^N +
\varepsilon_2(R^{N-1}),\qquad  \int_{\R^N}|\nabla w_R(|x|)|^{2} dx
= \varepsilon_3(R^{N-1}),
$$
and
$$
\int_{\R^N}|w_R(|x|)|^2|\nabla w_R(|x|)|^{2} dx =
\varepsilon_4(R^{N-1}),
$$
where $\varepsilon_i(R^{N-1})/R^N \to 0$, as $R \to \infty$, for
any $i= 2,3,4$. Thus letting $R \to \infty$ we have $\|w_R\|_2^2
\to +\infty$ and ${\mathcal E}(w_R)  \to - \infty$. This proves
Point (3).
\end{proof}

\vskip3pt\noindent In preparation to the proof of Theorem~\ref{no-minimizer} we
also show the following.

\begin{lemma}\label{lagrange-stat}
Assume that $1 < p < 3 + \frac{4}{N}$ and that $u_c \in X$ is a
minimizer of~\eqref{0.1} for some $c>0$. Then $u_c \in X$ weakly
satisfies
\begin{equation}
    \label{full-equation}
 - \Delta u_c - \lambda_c u_c - u_c\Delta |u_c|^2 =
|u_c|^{p-1}u_c
\end{equation}
with the  Lagrange multiplier $\lambda_c \in \R$ being strictly
negative.
\end{lemma}

\begin{proof} It is standard to show that $u_c \in X$
satisfies~\eqref{full-equation} for $\lambda_c \in \R$ being the
associated Lagrange multiplier, namely
\begin{equation}
     \label{lagrange}
    {\mathcal E}'(u_c)=\lambda_c u_c.
\end{equation}
Now applying Pohozaev identity to~\eqref{lagrange} yields
$$
\frac{1}{p+1} \int_{\R^N} |u_c|^{p+1} dx =
\frac{N-2}{N}\left[\frac{1}{2}\int_{\R^N} |\nabla u_c|^2 dx +
\int_{\R^N} |u_c|^2 |\nabla u_c|^2 dx \right] -
\frac{\lambda_c}{2} \int_{\R^N} |u_c|^2 dx.
$$
Thus, we obtain
$$
{\mathcal E}(u_c) = \frac{1}{N}\int_{\R^N} |\nabla u_c|^2 + 2
|u_c|^2 |\nabla u_c|^2 dx  + \frac{\lambda_c}{2} \int_{\R^N}
|u_c|^2 dx.
$$
Since ${\mathcal E}(u_c) \leq 0$, see Lemma~\ref{negative}, we
deduce that $\lambda_c <0$.
\end{proof}

We can now give the

\vskip4pt \noindent {\it Proof of Theorem~\ref{no-minimizer}.} To
prove i) we assume by contradiction that there exist a sequence
$(c_n) \subset \R^+$ with $c_n \to 0$ as $n\to\infty$ and $(u_n)
\subset X$ such that $m(c_n)$ is reached by $u_n \in X$. Then we
know, from Lemma~\ref{negative}, Point (2), that ${\mathcal
E}(u_n)\leq 0$, for all $n\in\N$ and using~\eqref{1.1}, we get
\begin{equation}\label{estimate1}
  \int_{\R^N}|u_n|^2 |\nabla u_n|^2 dx \leq K \Big(
\int_{\R^N}|u_n|^2 |\nabla u_n|^2 dx\Big)^{\frac{\theta
N}{N-2}}||u_n||_2^{2- 2 \theta}.
\end{equation}
If $p = 3 + \frac{4}{N}$ we have $\frac{\theta N}{N-2} =1$  and $2- 2\theta =
\frac{4}{N} >0$. Thus, since $||u_n||_2 \to 0$, we immediately get a
contradiction from~\eqref{estimate1}. Now if $p < 3 + \frac{4}{N}$, we have
$\frac{\theta N}{N-2} <1$ and thus,
\begin{equation}\label{tozero}
\int_{\R^N}|u_n|^2 |\nabla u_n|^2 dx \to 0,\quad\text{as $n\to\infty$}.
\end{equation}
Still using~\eqref{1.1}, we see from~\eqref{tozero} that
$\|u_n\|_{p+1} \to 0$ as $n\to\infty$. In turn, also
\begin{equation}
    \label{gradient}
\|\nabla u_n \|_2 \to 0,\quad\text{as $n\to\infty$},
\end{equation}
since ${\mathcal E}(u_n) \leq 0$ implies that $$\|\nabla u_n\|_2^2
\leq \frac{2}{p+1}\|u_n\|_{p+1}^{p+1}, \quad \mbox{ for all }
n\in\N.$$ At this point we distinguish two cases. First assume
that $1 + \frac{4}{N} \leq p \leq \frac{N+2}{N-2}$ if $N \geq 3$,
$1 + \frac{4}{N} \leq p$ if $N=1,2$. By H\"older and Sobolev
inequalities we have
$$
\|u_n\|_{p+1}^{p+1} \leq K(p,N) \|\nabla
u_n\|_2^{\frac{N}{2}(p-1)}||u_n||_2^{p+1 - \frac{N}{2}(p-1)}.$$ Since
${\mathcal E}(u_n)\leq 0$ it follows that
\begin{equation}\label{estimates2}
\|\nabla u_n\|^2_2 \leq \frac{2}{p+1}\int_{\R^N}|u_n|^{p+1}dx \leq
K(p,N)\|\nabla u_n\|_2^{\frac{N}{2}(p-1)}||u_n||_2^{p+1- \frac{N}{2}(p-1)}.
\end{equation}
If $p = 1 + \frac{4}{N}$ we have $\frac{N}{2}(p-1) = 2$ and $p+1 -
\frac{N}{2}(p-1) >0$. Thus we get directly a contradiction since
$||u_n||_2 \to 0$. If $1 + \frac{4}{N} < p \leq \frac{N+2}{N-2}$
for $N \geq 3$ and $1 + \frac{4}{N}<p$ for $N=1,2$ we have
$\frac{N}{2}(p-1) >2$ and $p+1- \frac{N}{2}(p-1) \geq 0$. Thus
there exists a $d >0$ such that $||\nabla u_n||_2 \geq d$ for all
$n \in \N$, yielding a contradiction with~\eqref{gradient}.
\medskip

Now we treat the remaining case $\frac{N+2}{N-2} <p < 3 + \frac{4}{N}$ with $N
\geq 3$. First, let us show that for any $q \geq \frac{4N}{N-2}$ the sequence
$(u_n) \subset X$ belongs to $L^q(\R^N)$ and it is uniformly bounded in
$L^q(\R^N)$. For this we follow a Moser's iteration argument presented in
the proof of~\cite[Lemma 5.10]{LiWaWa2}. Since $u_n \in X$ is a minimizer
for~\eqref{0.1} with $c = c_n$ we know, by Lemma~\ref{lagrange-stat}, that $u_n
\in X$ weakly satisfies~\eqref{full-equation}. Namely that
\begin{equation*}
\int_{\R^N}(1+ 2|u_n|^2) \nabla u_n \cdot \nabla \phi  + 2u_n |\nabla u_n|^2
\phi \, - \lambda_n u_n \phi \, - |u_n|^{p-1}u_n \phi \, dx = 0,
\end{equation*}
where $\lambda_n <0$ is the Lagrange parameter and $\phi \in
C_0^{\infty}(\R^N, \R)$. By an approximation argument, it is
easily seen that we can take as test functions any function in
$X$ which satisfies
$$
\int_{\R^N} u^2 |\nabla \phi|^2 dx < \infty \quad \mbox{ and }
\quad \int_{\R^N}|\nabla u|^2 \phi^2 \, dx < \infty.
$$ In
particular, setting $q_0 = \frac{4N}{N-2}$, we can choose as test
function, for any $M >0$ and any fixed $n \in \N$, $\phi_n^M =
|u_n^M|^{q_0 - p -1}u_n^M$ where $u_n^M = u_n$ when $|u_n(x)| \leq
M$ and $u_n^M = \pm M$ when $u_n \geq \pm M$. We then have, since
$|u_n^M| \leq |u_n|$ and $|\nabla u_n^M| \leq |\nabla u_n|$ for
any $n \in \N$, $M >0$, and using the fact that $\lambda_n <0$,
that
$$
(q_0 -p) \int_{\R^N}(1 + 2 |u_n^M|^2) |u_n^M|^{q_0 -p -1}|\nabla
u_n^M|^2 dx \leq \int_{\R^N} |u_n|^{q_0}dx.
$$ 
Since $q_0- p >1$
we have, in particular,
$$
2 \int_{\R^N} |u_n^M|^{q_0 - p +1} |\nabla u_n^M|^2 dx \leq
\int_{\R^N}|u_n|^{q_0}\, dx.
$$ 
Finally, for $n \in \N$ fixed,
letting $M \to + \infty$ we obtain that
\begin{equation}
\label{controlbbis}
 2 \int_{\R^N} |u_n|^{q_0 - p +1} |\nabla u_n|^2 dx \leq
\int_{\R^N}|u_n|^{q_0}\, dx.
\end{equation}
Now, notice that, by Sobolev inequality,
\begin{equation*}
2 \int_{\R^N} |u_n|^{q_0 - p +1} |\nabla u_n|^2 dx = L(p,N) \int_{\R^N}
\big|\nabla  |u_n|^{r} \big|^2 dx  
\geq \tilde L(p,N) \|u_n^{r}\|_{\frac{2N}{N-2}}^2,
\end{equation*}
for some constants $L,\tilde L>0$, and where $$r = \frac{q_0 - p +3}{2}.$$ Thus $(u_n) \subset
L^{\frac{2Nr}{N-2}}(\R^N)$ and since, by~\eqref{tozero}, $(u_n) \subset
L^{q_0}(\R^N)$ is bounded, by~\eqref{controlbbis}, $(u_n) \subset
L^{\frac{2Nr}{N-2}}(\R^N)$  is also bounded. Since $ p < \frac{3N+2}{N-2}$ 
it follows that
$$
\frac{2Nr}{N-2} > q_0
$$ and the Moser iteration can be
continued further on. Thus, we obtain that $(u_n) \subset L^q(\R^N)$ for any
$q \geq q_0$ with $(u_n) \subset L^q(\R^N)$  bounded. At this
point, by H\"older and Sobolev inequalities we can write
\begin{equation}\label{match}
 \|u_n\|_{p+1}^{p+1} \leq C(p,N) \|\nabla u_n\|_2^{\alpha}\,
\|u_n\|_{(p-1)N}^{\beta} \end{equation} with 
$$
\alpha = \frac{2N(p-1) -
2(p+1)}{(p-1)(N-2)-2} \qquad \mbox{ and } \qquad \beta = N(p-1) \frac{(N-2) (p+1)
- 2N }{(p-1)N(N-2) - 2N}.
$$ Now as in~\eqref{estimates2}, using the fact that
${\mathcal E}(u_n) \leq 0$, we get that
$$
\|\nabla u_n\|_2^2 \leq K(p,N)\|\nabla u_n\|_2^{\alpha}\,
\|u_n\|_{(p-1)N}^{\beta}.$$ As $p >1$ we have $\alpha >2$ and since $(u_n)
\subset L^{(p-1)N}(\R^N)$ is bounded we obtain again a contradiction
with~\eqref{gradient}. Notice that, in~\eqref{match}, the coefficient $(p-1)N$ as
no particular meaning, we just choose it sufficiently large in order to insure
that, in turn, $\alpha >2$.  This proves Point i) since if $m(c) <0$ a minimizer always
exists by Lemma~\ref{general-minimizer}. \smallskip

For the proof of Point ii) we know from, Lemma~\ref{negative},
Point (3), that there exists a $c>0$ such that $m(c) <0$. Now let
$d>0$ be such that $m(d) <0$ and $u \in X$ be an associated
minimizer. We consider the scaling $v(x) = u(\sigma^{-
\frac{1}{N}}x)$ used in the proof of
Lemma~\ref{general-minimizer}. For $\sigma
>1$ we have $\|v\|_2^2 >d$ and ${\mathcal E}(v) < {\mathcal
E}(u)$. This proves the claim. We also point out that very likely
the function $\{c \to m(c)\}$ is continuous for $c>0$ so that also
$m(c(p,N))=0$. However we did not pursue in that direction.\qed

\section{Orbital stability}
\label{stability}

In this section we prove the orbital stability result, Theorem
\ref{stab}. The key point is to prove the following:

\begin{lemma}
        \label{all-converge}
Assume that $1 < p < 3 + \frac{4}{N}$ and that $c>0$ is such that
$m(c) <0$. Then for any real minimizing sequence of~\eqref{0.1},
there exists a subsequence that is strongly converging in $X$, up
to a translation in $\R^N$.
\end{lemma}
\begin{proof}
Let $(u_n)\subset X$ be any minimizing sequence for
problem~\eqref{0.1}. We will prove the assertion by means of
Lions's Compactness-Concentration Principle
(cf.~\cite{pll-1,pll-2}), applied to the  sequence
$$
\rho_n(x)=u_n^2(x),\qquad n\in\N.
$$
Let us first proves that the vanishing, namely
$$
\sup_{y \in \R^N}\int_{y + B_R}|u_n|^2 dx \to 0 \quad \mbox{for
all } R >0,
$$
cannot occur. By Lemma~\ref{boundedness}, we know that
\begin{equation} \label{borne}
(u_n) \subset X \, \mbox{ is bounded  and } \,
\int_{\R^N}|\nabla(u_n^2)| dx \mbox{ is bounded in } \R.
\end{equation}
We apply ~\cite[Lemma I.1]{pll-2} to the sequence $\rho_n$.
Indeed, $\rho_n$ is bounded in $L^1(\R^N)$ and $\nabla \rho_n$ is
bounded in $L^2(\R^N)$. Then for every $\alpha$ such that $1\leq
\alpha\leq \frac{2N}{N-2}$, $\rho_n \to 0$ in
$L^\alpha(\R^N)$, as $n$ goes to $\infty$. Taking
$\alpha=\frac{p+1}{2}$ (this choice is valid since $1<p<3+\frac{4}{N}$) provides
$$
\|\rho_n\|_{\frac{p+1}{2}}=\|u_n\|_{p+1}^2\to 0\quad\text{as $n\to\infty$},
$$
and then $\liminf_{n\to\infty} \mathcal{E}(u_n)\geq 0$, which contradicts the
fact that $m(c) <0$. Now, by following the lines of the proof
of~\cite[Lemma III.1]{pll-1}, one can show that there exists a
subsequence $u_{n_k}$ (that we will still denote by $(u_n)$) such
that either compactness occurs or dichotomy occurs in the
following sense: there exists $\alpha\in(0,c)$ such that, for all
$\eps>0$, there exists $k_0\geq 1$ and two sequences
$(u_n^1),(u_n^2)$ bounded in $X$ such that, for all $k\geq k_0$,
\begin{align}
        \label{prima}
& \|u_n-(u_n^1+u_n^2)\|_{L^{p+1}}\leq \delta(\eps),
\quad\text{$1 < p < 3 + \frac{4}{N}$,\quad with $\delta(\eps)\to 0$ for $\eps\to 0$,} \\
& \left|\int_{\R^N}(u_n^1)^2dx-\alpha\right|\leq\eps,\quad
\left|\int_{\R^N}(u_n^1)^2dx-(c-\alpha)\right|\leq\eps,  \notag \\
& {\rm dist}({\rm supp}\,u_n^1,{\rm supp}\,u_n^2)\to
\infty,\quad\text{as $n\to\infty$},
\notag \\
& \liminf_{k\to\infty}\int_{\R^N}(|\nabla u_n|^2-|\nabla
u_n^1|^2-|\nabla u_n^2|^2)dx\geq 0,
\label{quattro} \\
& \liminf_{k\to\infty}\int_{\R^N}(|\nabla (u_n)^2|^2-|\nabla
(u_n^1)^2|^2-|\nabla (u_n^2)^2|^2)dx\geq 0. \label{cinque}
\end{align}
We point out that, only inequalities~\eqref{prima}
and~\eqref{cinque} have to be proved, the other inequalities are
already contained in~\cite[Lemma III.1]{pll-1}. Because
of~\eqref{borne} and taking into account inequality~\eqref{1.1} we
learn that there exists a positive constant $K$ such that, for all
$n\in\N$,
\begin{equation}
        \label{controllo-p-2}
    \int_{\R^N}|u_n|^{p+1} dx \leq  K \Big( \int_{\R^N} |u_n|^2 dx\Big)^{1 - \theta},
\qquad \theta = \frac{(p-1)(N-2)}{2(N+2)}.
\end{equation}
Thus, inequality~\eqref{prima} follows from the corresponding
inequality for the $L^2$ norm which is contained in the proof
of~\cite[Lemma III.1]{pll-1}. Now inequality~\eqref{cinque} can be
obtained by arguing as for the proof of~\eqref{quattro}. Indeed,
notice that, if $\varphi_R$ is a given smooth cut-off function,
$0\leq\varphi_R\leq 1$, $\varphi_R=1$ on $B(0,R)$, $\varphi_R=0$
outside $B(0,2R)$ and $|\nabla \varphi_R|\leq \frac{1}{R}$, and
$v_n$ is a sequence in $X$ satisfying the boundedness
condition~\eqref{controllo-p-2}, then we have
\begin{align*}
|\nabla (\varphi_R v_n)^2|^2-\varphi_R^4|\nabla v^2_n|^2
&=4\varphi_R^3v_n^2\nabla\varphi_R\cdot\nabla v_n^2
+4\varphi_R^2|\nabla\varphi_R|^2v_n^4 \\
&\leq      2\varphi_R^3|\nabla\varphi_R|v_n^4
+2\varphi_R^3|\nabla\varphi_R||\nabla v_n^2|^2
+4\varphi_R^2|\nabla\varphi_R|^2v_n^4,
\end{align*}
for all $n\geq 1$, yielding
$$
\left| \int_{\R^N}|\nabla (\varphi_R v_n)^2|^2 dx-
\int_{\R^N}\varphi_R^4|\nabla v^2_n|^2dx
\right|\leq\frac{C}{R},\quad\text{for all $n\geq 1$},
$$
for some positive constant $C$ independent of $n$. This last
inequality is therefore sufficient to obtain Inequality
\eqref{cinque}.

Now, it is standard to see that if the dichotomy property holds
(with the inequalities indicated above), then sending $\eps$ to
zero, the following inequality holds true
\begin{equation*}
m(c)\geq m(\alpha) + m(c- \alpha).
\end{equation*}
To conclude we now prove that instead we have, for any $c_1,c_2>0$
such that $c_1 +c_2 =c$,
\begin{equation}
        \label{dichotomy}
m(c)< m(c_1) + m(c_2).
\end{equation}
In light of~\cite[Lemma II.1]{pll-1}, to show
that~\eqref{dichotomy} holds, it is sufficient to prove that, for
any $d>0$ such that $m(d) <0$,
\begin{equation}
        \label{sur-dichotomy}
m(\lambda d) < \lambda m(d),\quad\text{for any $\lambda>1$}.
\end{equation}
To prove inequality~\eqref{sur-dichotomy} we observe that, if $u_d
\in X$ is a minimizer of $m(d)$, then setting $v(x) =
u_d(\lambda^{- \frac{1}{N}}x)$ we have $\|v\|_2^2 = \lambda d$ and
\begin{align*}
{\mathcal E}(v) & =  \lambda^{1 - \frac{2}{N}} \Big[ \int_{\R^N}
\frac{1}{2} |\nabla u_d|^2  + |u_d|^2 |\nabla u_d|^2 dx
\Big]- \lambda \int_{\R^N}|u_d|^{p+1} dx  \\
& =  \lambda \Big[ \lambda^{ - \frac{2}{N}}
\int_{\R^N}\big(\frac{1}{2} |\nabla u_d|^2  +|u_d|^2 |\nabla
u_d|^2
dx\big) -  \int_{\R^N}|u_d|^{p+1}dx \Big] \\
& <   \lambda  m(d).
\end{align*}
Thus ${\mathcal E}(v) < \lambda m(d)$ which lead to $m(\lambda d)
< \lambda m(d)$, proving the claim. \medskip

Since we ruled out both vanishing and dichotomy, we have
compactness for $\rho_n$, namely we know that there exists a
sequence $(y_n)\subset \R^N$ such that, for any $\eps>0$, there is
$R>0$ with
\begin{equation}\label{compact}
\int_{y_n +B_R}|u_n|^2 dx \geq c - \varepsilon.
\end{equation}
We then denote $\tilde{u}_n= u_n (\cdot + y_n)$ and clearly from
inequality~\eqref{compact} we have $\tilde{u}_n \to \tilde{u}$
strongly in $L^2(\R^N)$, as $n\to\infty$. By~\eqref{controllo-p-2}
we then see that $\tilde{u}_n \to \tilde{u}$ strongly in
$L^p(\R^N)$. At this point, taking into account Point 1) of
Lemma~\ref{weak-convergence}, and since $\tilde{u}_n
\rightharpoonup \tilde{u}$ in $X$, we get that ${\mathcal
E}(\tilde{u}) \leq \liminf {\mathcal E}(\tilde{u_n}) = m(c).$ This
proves that $\tilde{u} \in X$ minimize~\eqref{0.1} and then,
necessarily, $\nabla \tilde{u}_n \to \nabla u$ in $L^2(\R^N)$, as
$n\to\infty$, proving the strong convergence of $\tilde{u}_n$ to
$\tilde{u}$ in $X$. This concludes the proof.
\end{proof}
\medskip

Now we can give the
\vskip3pt
\noindent {\it Proof of
Theorem~\ref{stab}.} First note that if $(u_n)$ is a minimizing
sequence for ~\eqref{0.1}, then $(|u_n|)$ is also a minimizing
sequence and is real. Then by Lemma~\ref{all-converge}, there
exists a subsequence $(|u_{n_k}|)$ of $(|u_n|)$ and a sequence
$(y_{n_k}) \subset \R^N$ such that $(|u_{n_k}(\cdot - y_{n_k})|)$
converges strongly in $H^1(\R^N)$ toward $u$ where $u$ is real and
solves~\eqref{0.1}. Then the result follows by standard
considerations (see, for example,~\cite{CaLi}). \qed

\begin{remark}\label{lien}
Take any solution $u$ to problem~\eqref{defvalc}, namely
$\|u\|_2^2=c$ and $\mathcal{E}(u)=m(c)$. Then it is a classical
fact that there exists a parameter $\omega^*$, depending on $c$
and $u$, such that $u$ solves equation~\eqref{gs1} with
$\omega=\omega^*$ (see Lemma~\ref{lagrange-stat}). If aim to study
the orbital stability issue of the ground states of~\eqref{gs1}
via the constrained approach (as it is the case in the classical
paper of Cazenave-Lions~\cite{CaLi}) we need to have more
informations on the ground states of~\eqref{gs1}. In particular we
need to know that they share the same $L^2$ norm. Except when
$N=1$ where we have the uniqueness of the ground states, this
information is not available to us. Now, when $N=1$ we still need
to know if, when $u_1$ and $u_2$ are two distinct solutions to the
minimization problem~\eqref{defvalc}, then we have
$\omega^*_1=\omega^*_2$. We did not manage to show this.
\end{remark}

\begin{remark}\label{connections}
In~\cite{BoCh} two results about orbital stability are presented.
When $N \geq 2$, in Theorem 3.3, assuming that $1 <p < 1+
\frac{4}{N}$ a result of orbital stability  within the class of
radial functions, is announced for the minimizers
of~\eqref{defvalc}. However, the proof is incorrect and, we guess,
can be fixed only assuming in addition that $p\geq 3$. When $N=1$,
$4\leq p <6$, Proposition 4.7 guarantees the orbital stability of
the minimizers of~\eqref{defvalc} but Theorem 4.8, which establish
the connection with the problem where $\omega >0$ is fixed, is
false because the uniqueness of the Lagrange parameter is not
known (see Remark~\ref{lien}).
\end{remark}

\bigskip
\bigskip

\end{document}